\documentclass[a4paper,11pt]{amsart}

\usepackage{amssymb}
\usepackage{epsfig}
\usepackage{amsfonts}
\usepackage{amsmath}
\usepackage{euscript}
\usepackage{amscd}
\usepackage{amsthm}
\usepackage{bbm}
\usepackage[all,cmtip]{xy}
\DeclareMathAlphabet{\mathpzc}{OT1}{pzc}{m}{it}

\usepackage{soul}
\setstcolor{red}

\makeatletter
\@namedef{subjclassname@2020}{%
  \textup{2020} Mathematics Subject Classification}
\makeatother

\usepackage[colorinlistoftodos,textwidth=2cm,textsize=small]{todonotes}
\usepackage{verbatim}
\usepackage[pagebackref,hypertexnames=false, colorlinks, citecolor=red,linkcolor=blue, urlcolor=red]{hyperref}
\usepackage{mathrsfs}
\usepackage{tikz}
\newcommand*\circled[1]{\tikz[baseline=(char.base)]{
            \node[shape=circle,draw,inner sep=0.2pt] (char) {#1};}}
\DeclareSymbolFont{SY}{U}{psy}{m}{n}
\DeclareMathSymbol{\emptyset}{\mathord}{SY}{'306}

\theoremstyle{plain}

\newtheorem{thm}{Theorem}[section]
\newtheorem{cor}[thm]{Corollary}
\newtheorem{lem}[thm]{Lemma}
\newtheorem{prop}[thm]{Proposition}
\theoremstyle{definition}
\newtheorem{defn}[thm]{Definition}
\newtheorem{rem}[thm]{Remark}

\numberwithin{equation}{section}

\def\A{{\mathcal A}}
\def\C{{\mathbb C}}
\def\D{{\mathbb D}}
\def\B{{\mathbb B}}
\def\S{{\mathbb S}}

\def\O{{\mathcal O}}
\def\H{{\mathcal H}}

\def\K{{\mathscr K}}

\def\dbar{\bar\partial}

\def\inp#1,#2{\left\langle{#1},{#2}\right\rangle}
\def\bf#1{{\textbf{#1}}}
\def\N{\mathbb{N}}

\def\R{\mathbb{R}}
\def\A{\mathbb{A}}
\def\K{\mathbb{K}}

\def\si{\sigma}

\def\l{\lambda}

\def\O{\Omega}

\def\ra{\rightarrow}
\def\ov{\overline}

\def\a{\alpha}
\def\b{\beta}

\def\G{\Gamma}

\def\h{ hermitian holomorphic vector bundle}

\def\w{with respect to }

\def\beq{\begin{eqnarray}}
\def\eeq{\end{eqnarray}}
\def\beqa{\begin{eqnarray*}}
\def\eeqa{\end{eqnarray*}}

\def\del{\partial}
\def\T{\boldsymbol{T}}
\def\M{\boldsymbol{M}}
\def\rkhs{reproducing kernel Hilbert space }
\def\<{\langle}
\def\>{\rangle}

\def\bz{\boldsymbol{z}}
\def\bw{\boldsymbol{w}}
\def\bf{\boldsymbol{f}}
\def\bg{\boldsymbol{g}}

\def\bA{\boldsymbol{A}}
\def\bH{\boldsymbol{H}}
\def\bl{\pmb{\lambda}}
\def\bmu{\pmb{\mu}}
\def\ba{\pmb{\alpha}}
\def\bb{\pmb{\beta}}

\newcommand{\be}{\begin{equation}}
\newcommand{\ee}{\end{equation}}
\newcommand{\bea}{\begin{eqnarray}}
\newcommand{\eea}{\end{eqnarray}}
\newcommand{\Bea}{\begin{eqnarray*}}
\newcommand{\Eea}{\end{eqnarray*}}

\newcounter{cnt1}
\newcounter{cnt2}
\newcounter{cnt3}
\newcommand{\blr}{\begin{list}{$($\roman{cnt1}$)$}
 {\usecounter{cnt1} \setlength{\topsep}{0pt}
 \setlength{\itemsep}{0pt}}}
\newcommand{\bla}{\begin{list}{$($\alph{cnt2}$)$}
 {\usecounter{cnt2} \setlength{\topsep}{0pt}
 \setlength{\itemsep}{0pt}}}
\newcommand{\bln}{\begin{list}{$($\arabic{cnt3}$)$}
 {\usecounter{cnt3} \setlength{\topsep}{0pt}
 \setlength{\itemsep}{0pt}}}
\newcommand{\el}{\end{list}}
\newcommand{\overbar}[1]{\mkern 1.5mu\overline{\mkern-1.5mu#1\mkern-1.5mu}\mkern 1.5mu}

\DeclareMathOperator{\Aut}{Aut}\DeclareMathOperator{\aut}{Aut}
\DeclareMathOperator{\mob}{\text{M\"ob}}

\newcounter{defcounter}
\begin{document}
\title[Homogeneous vector bundles and operators]{Homogeneous Hermitian Holomorphic  Vector Bundles And  Operators In The Cowen-Douglas Class Over The Poly-disc}
\author[P. Deb]{Prahllad Deb}
\address[P. Deb]{Department of Mathematics, Ben-Gurion University of the Negev, Beer-Sheva, Israel-84105.}
\email[P.Deb]{prahllad.deb@gmail.com} 
\author[S. Hazra]{Somnath Hazra}
\address[S. Hazra]{Mathematics Institute, Silesian University in Opava, Na Rybnicku 626/1, 74601, Opava.}
\email[S. Hazra]{somnath.hazra.2008@gmail.com} 

\keywords{Cowen-Douglas class, Homogeneous operators, Hermitian holomorphic homogeneous vector bundles, Curvature, Representation, Lie algebra, Lie group} 
\subjclass[2020]{Primary: 47B32, 47B13, 17B10 Secondary: 20C25, 53C07}
\thanks{The research of the first named author was supported through IISER Kolkata Ph.D Fellowship, J C Bose National Fellowship of G. Misra followed by a post-doctoral fellowship at Ben-Gurion University of the Negev, Israel. The research of the second named author was supported through research Fellowships of CSIR, IISc, NBHM post-doctoral Fellowship followed by GA CR grant  No. 21-27941S. Some of the results in this paper are from the PhD thesis of the first named author submitted to the Indian Institute of Science Education and Research Kolkata and from the PhD thesis of the second named author submitted to the Indian Institute of Science. }

\begin{abstract} In this article, we obtain two sets of results. The first set of results are for the case of the bi-disc while the second set of results describe in part, which of these carry over to the general case of the poly-disc. 
\begin{enumerate}
\vspace{0.05in}
\item [] A classification of irreducible hermitian holomorphic vector bundles over $\mathbb{D}^2$, homogeneous with respect to M\"ob$\times $M\"ob, is obtained assuming that the associated representations are \textit{multiplicity-free}. 
Among these the ones that give rise to an operator in the Cowen-Douglas class of $\mathbb{D}^2$ of rank $1,2$ or $3$ are determined.  

\vspace{0.03in}
\item[] Any hermitian holomorphic vector bundle of rank $2$ over $\D^n$, homogeneous with respect to the $n$-fold direct product of the group $\mob$ is shown to be a tensor product of $n$ hermitian holomorphic vector bundles over $\D$. Among them, $n-1$ are shown to be the line bundles and one is shown to be a rank $2$ bundle. Also, each of the bundles are homogeneous with respect to $\mob$.

\vspace{0.03in}
\item[] The classification of irreducible homogeneous hermitian holomorphic vector bundles over $\mathbb{D}^2$ of rank $3$ (as well as the corresponding Cowen-Douglas class of operators) is extended to the case of $\mathbb{D}^n$, $n>2$. 
\vspace{0.03in}
\item[] It is shown that there is no irreducible $n$ - tuple of operators in the Cowen-Douglas class  $\mathrm B_2(\mathbb{D}^n)$ that is homogeneous with respect to Aut$(\mathbb{D}^n)$, $n >1$. 
Also, pairs of operators in $\mathrm B_3(\mathbb{D}^2)$  homogeneous with respect to Aut$(\mathbb{D}^2)$ are produced, while it is shown that no $n$ - tuple of operators in $\mathrm B_3(\mathbb{D}^n)$ is homogeneous with respect to Aut$(\mathbb{D}^n)$, $n > 2$. 
\end{enumerate} 

\end{abstract}


\maketitle


\section{Introduction}
We let $\mob$ denote the bi-holomorphic automorphism group of the unit disc $\D:=\{z\in \mathbb C: |z| < 1\}$. A bounded linear operator on a complex separable Hilbert space $\H$ is said to be homogeneous if the spectrum $\si(T)$ of $T$ is contained in $\overbar{\D}$ and, for every $g\in \mob$, $g(T)$ is unitarily equivalent to $T$. The condition $\si(T) \subseteq \overbar{\D}$ ensures that $g(T)$ is well defined for every $g \in \mob$. Indeed, since if $g \in \mob$, then $g(z) = e^{i \theta} \frac{z - a}{1 - \bar{a}z}$ for some $\theta \in \mathbb{R},\,\,a, z \in \D$ and therefore, $g(T) = e^{i \theta} (T - aI)(I - \bar{a}T)^{-1}$ is well defined. 

The class of homogeneous operators has been studied in a number of articles \cite{THS, HOPRMGS, HOHS, ACHOCD, HOJCS, OICHO, SH}. 
The classification of irreducible homogeneous operators in the Cowen-Douglas class over $\D$ has been completed recently (cf. \cite{ACHOCD}).
The notion of a homogeneous operator has a natural generalization to commuting tuples of operators. A commuting $n$ - tuple of operators $(T_1, T_2,\ldots ,T_n)$ is said to be homogeneous with respect to a subgroup $G$ of the group of bi-holomorphic automorphisms $\aut(\D^n)$ of $\D^n$, if  the joint spectrum of $(T_1, T_2,\ldots , T_n)$ lies in $\overline{\mathbb{D}}^n$ and $g(T_1, T_2,\ldots , T_n)$ is unitarily equivalent with $(T_1, T_2,\ldots ,T_n)$ for all $g \in G$. Let $\mob^n$ denote the direct product of $n$-copies of $\mob$. Note that $\aut(\D^n)$ is the semi-direct product of $\mob^n$ and the permutation group $S_n$ of $n$-elements. In this generality, the focus has been on the study of commuting tuples of operators in the Cowen-Douglas class which are homogeneous with respect to either $\mob^n$ or $\aut(\D^n)$. First, we recall the definition of the Cowen-Douglas class. This important class consisting of bounded operators was introduced in \cite{CGOT}, then modified to include pairs of operators in  \cite{OPOSE} and finally in \cite{GBKCD} a class of commuting $n$ - tuple of bounded operators with similar properties was introduced. We reproduce the definition following \cite{GBKCD}.

Let $\O\subset\C^n$ be a bounded domain and $D_{\T} : \H \to \H \oplus \cdots \oplus \H$ be the operator given by the formula: $D_{\T} h = \left(T_1 h,\hdots,T_n h\right)$, $h \in \H.$ A commuting $n$ - tuple of bounded linear operators $\T=(T_1,\hdots,T_n)$ on $\H$ is said to be in $\mathrm B_r(\Omega)$ if 
\begin{itemize} 
\item $\dim \ker D_{\T - \bz I}=r$, $\bz\in \Omega$;
\item  $\text{ran} D_{\T - \bz I}$ is closed in $\H\oplus\cdots\oplus\H$;
\item the linear span of the vectors in $\ker D_{\T - \bz I}$, $\bz\in  \Omega$ is dense in $\mathcal H$. 
\end{itemize}

Following the ideas of \cite{OPOSE}, it is easy to  establish a one to one correspondence between the unitary equivalence class of commuting $n$ - tuples in $\mathrm B_r(\D^n)$ and  equivalence class of  the corresponding  {\h} $E$ over $\D^n$. The equivalence class of the {\h} is the local equivalence of the hermitian structure. These vector bundles are distinguished, among others, by the property that the hermitian structure on the fibre over $\bz\in \D^n$ is induced from the inner product of a fixed Hilbert space $\mathcal H$.  It has been proved in \cite{GBKCD} that the corresponding $n$ - tuple of operators $\T$ is simultaneously unitarily equivalent to the adjoint of the $n$ - tuple of multiplication operators $\M=(M_{z_1},\hdots,M_{z_n})$ by the coordinate functions on a Hilbert space $\H_K$ possessing a reproducing kernel $K$. 

An automorphism of a vector bundle $\pi: E \to \Omega$ is a diffeomorphism $\hat{g}:E \to E$ such that $\pi\circ \hat{g} = g \circ \pi$ for some automorphism $g:\Omega \to \Omega$. The bundle map $\hat{g}$ is a lift of $g$. We say that a holomorphic hermitian vector bundle $E$ is homogeneous \w a subgroup $G$ of the bi-holomorphic automorphism group of $\Omega$ if the action of $G$ on $\O$ is transitive and the lift of this action is isometric with respect to the hermitian structure of the bundle. 

The correspondence between unitary equivalence classes of operators $\boldsymbol T$ in $\mathrm B_r(\Omega)$ and the local hermitian equivalence classes of hermitian holomorphic vector bundles $E_{\boldsymbol T}$ on $\Omega$  determined by $\boldsymbol T$ extends to homogeneous operators and homogeneous vector bundles. In this paper, we exploit this correspondence in the case of $\Omega=\mathbb D^n$ using the explicit representation theory of the identity component of the automorphism group of $\mathbb D^n$. 

Thus the  classification of homogeneous $n$ - tuples of operators in $\mathrm B_r(\D^n)$ is the same as that of the classification  of homogeneous {\h}s over $\D^n$ such that the hermitian structure is induced from a reproducing kernel. 

\textit{Main Question:} {What are the irreducible $n$ - tuples of  operators in $\mathrm B_r(\D^n)$, $r\geq 1$ and $n>1$, which are homogeneous \w $\aut(\D^n)$, or the subgroup $\mob^n$ of $\aut(\D^n)$? There will be a slight advantage if we replace the groups by their universal covering groups while noting that it makes no difference to the question of homogeneity.} 

Homogeneous operators in $\mathrm B_r(\Omega)$ \w the automorphism group of an irreducible bounded symmetric domain $\Omega$ have been studied in \cite{HHVCDBSYMD, HVBIOSD, OFGAB}.

Following the general principles outlined in \cite{ACHOCD, HHVCDBSYMD}, it is evident that the irreducible homogeneous holomorphic vector bundles over $\mathbb D^n$ of rank $r$ are in one to one correspondence with $r$ dimensional indecomposable linear representations of the $n$-fold direct sum $\mathfrak{b}^n$ of the solvable Lie subalgebra $\mathfrak{b} \subseteq \mathfrak{sl}(2,\C)$ consisting of lower triangular matrices. Among these, we want only those which are skew-hermitian on the sub-algebra $\mathfrak k$, the useful property of these being that $\rho(h)$ for any   $h$ in the sub-algebra $\mathfrak k$ is diagonalizable. This gives the hermitian structure for the homogeneous holomorphic vector bundle. Thus the indecomposable $r$ dimensional linear representations of the Lie algebra $\mathfrak{b}^n$ with this additional property is one of the main ingredients of our proof.  These representations have been described for $n=1$ in \cite[Section 2, Page 6]{MFHOCD}. 
Let $\mathbbm 1$ be the trivial representation  and $\rho$ be an indecomposable representation  of $\mathfrak b$. For  $v_i \in \mathfrak b,\,\, 1\leq i \leq n,$ define the representation of $\mathfrak b^n$ by the rule $$\rho(v_1)  \otimes I \otimes \cdots \otimes I + I \otimes \mathbbm 1(v_2) \otimes I \otimes \cdots \otimes I + \cdots + I\otimes \cdots \otimes I \otimes \mathbbm 1(v_n)$$ 
acting on the $n$-fold tensor product of the representation spaces. This defines a family of indecomposable representations of $\mathfrak{b}^n$. 
It is natural to ask if these are all the indecomposable representations of $\mathfrak{b}^n$. We show that the answer is affirmative if $r=1$ or $2$, see Theorem \ref{Description of all indecomposable rep}.  Moreover, we describe all indecomposable \textit{multiplicity-free} (Definition \ref{mult free}) representations of $\mathfrak{b}^2$ explicitly in terms of representations of $\mathfrak b$, see Theorem \ref{Classification of rep}.

In the following section, we obtain a transformation rule for the curvature tensors of {\h}s over $\D^n$ homogeneous \w the universal covering group of $\mob^n$ which will be used in subsequent sections. Note that this formula (Proposition \ref{Transformation rule for curvature}), however, holds for any bounded domain $\O\subset \C^n$. 

All holomorphic line bundles over $\D^n$, which are homogeneous with respect to either the subgroup $\mob^n$ or $\aut(\D^n)$, are described in Section \ref{Section 3}. 
When this rank is more than $1$, determining the class of homogeneous tuples  in $\mathrm B_{r}(\mathbb{D}^{n})$  is much more difficult.  In this paper, a complete list of inequivalent homogeneous tuples in $\mathrm B_{r}(\mathbb{D}^{n})$ is given for $r=2$ and $r = 3$ in Theorem \ref{homogeneous rank 2 bundle} and Theorem \ref{thm17}, respectively, following the technique described in Section \ref{Section 4}. 


In Section \ref{Section 5}, all indecomposable \textit{multiplicity-free} (Definition \ref{mult free}) representations of $\mathfrak{b}\oplus\mathfrak{b}$ are described. Consequently, we obtain a complete characterization of all irreducible {\h}s over $\D^2$ which are homogeneous \w $\mob^2$ and whose associated representations are \textit{multiplicity-free}. This classification result is then used in Section \ref{Section 6} and \ref{Section 8} to determine all irreducible $n$ - tuple of operators in $\mathrm B_2(\D^n)$ and $\mathrm B_3(\D^n)$, respectively, which are homogeneous \w $\mob^n$. We note that while all such representations of dimension $2$ are obtained from tensoring $(n-1)$ one dimensional representations and a two dimensional representation of $\mathfrak b$, this is no longer true when the dimension is three, see Corollary \ref{thm15}. Also, the existence of an irreducible homogeneous hermitian holomorphic vector bundle of rank $3$ which is not equivalent to the tensor product of a homogeneous 
holomorphic hermitian line bundle on $\mathbb D^{n-1}$ and a homogeneous holomorphic hermitian vector bundle of rank 3 on the disc $\mathbb D$ follows from it. Thus, unlike the case of $\mathrm B_1(\D^n)$, or $\mathrm B_2(\D^n)$, there exist $n$ - tuples of homogeneous operators in $\mathrm B_3(\D^n)$ that can not be obtained  from the known list of homogeneous operators in $\mathrm B_r(\D)$ described in \cite{ACHOCD} by taking tensor product. 

In Section \ref{Section 7}, two families of homogeneous irreducible $n$ - tuple of operators in $\mathrm B_3(\D^n)$ are constructed one of which may be seen as a multi-variable analogue of the construction presented in \cite{HOHS}. However, the other family turns out to be a new one occurred only when $n>1$. It is then shown in Section \ref{Section 8} that this list is complete modulo unitary equivalence. Moreover, the homogeneity with respect to $\aut(\D^n)$ is also discussed at the end of Section \ref{Section 8}.

\section{Curvature formulae for Quasi-invariant Kernels}\label{Section 2}
Let $\aut(\D^n)$ be the group of all bi-holomorphic automorphisms of $\D^n$. For any subgroup $G$ of $\Aut(\D^n)$, let $\tilde{G}$ denote the universal cover of $G$. A reproducing kernel $K : \D^n \times \D^n \to \text{M}(r,\C)$ is said to be \textit{quasi-invariant} if there exists a family of holomorphic functions $J_{\tilde{g}} : \D^n \to \text{GL}(r, \C)$, $\tilde{g} \in \tilde{G}$ such that $K$ satisfies the transformation rule $$K(\bz, \bw) = J(\tilde{g}, \bz) K(g \bz,g \bw) \overline{J(\tilde{g}, \bw)},\,\,\bz, \bw \in \mathbb{D}^n,\,\,\tilde{g} \in \tilde{G}$$ where $g=p(\tilde{g})$ and $p:\tilde{G}\ra G$ is the universal covering map. In this section, we first describe the curvature of a quasi-invariant kernel. We also show that these kernels are intimately related to the study of homogeneous operators in the Cowen-Douglas class. 

Let $e_1, \ldots, e_r$ be the standard unit vectors in $\mathbb C^r.$ For $1\leq i \leq r,$  define $s_i:\mathbb D^n \to H_K$ to be the anti-holomorphic map:  $s_i(\bw):=K(\cdot,\bw) e_i,$ $\bw\in \mathbb D^n.$ Clearly, $(s_1, \ldots, s_r)$ defines a trivial anti-holomorphic hermitian vector bundle $E$ of rank $r$ on $\mathbb D^n.$  The fiber of $E$ at $w$ is the $r$ - dimensional subspace $\big \{K(\cdot,\bw) x:x\in \mathbb C^r \big \}$  and the hermitian structure at $\bw$ is given by the positive definite matrix $K(\bw,\bw).$ Thus the curvature $\mathsf K$  of the vector bundle $E$ is a $(1,1)$ form  given by the formula:
$$\mathsf K(\bw) = \sum_{i,j=1}^n \partial_{i} \left [ K(\bw, \bw)^{-1} \overline{\partial_{j}} K(\bw, \bw) \right ] d w_i \wedge d \bar{w}_j.$$
Although, not very common, we will let  
$$
\mathcal K(\bw)= \big ( \!\! \big (\mathcal K^{ij}(\bw) \big ) \!\! \big ), \,\,w\in \mathbb D^n,
$$ 
where $\mathcal K^{ij}(\bw) :=\partial_i \left [ K(\bw, \bw)^{-1} \overline{\partial_{j}} K(\bw, \bw) \right ]$ is the co-efficient of $d w_i \wedge d \bar{w}_j$ in $\mathsf K.$  
We obtain a transformation rule for the curvature whenever the kernel $K$ is quasi-invariant. 

\begin{prop}\label{Transformation rule for curvature}
Let $G$ be a subgroup of $\Aut(\D^n)$. Suppose $J_{\tilde{g}} : \mathbb{D}^n \rightarrow \text{GL}(r, \mathbb{C}),$ $\tilde{g}\in \tilde{G},$ is holomorphic and $K : \mathbb{D}^n \times \mathbb{D}^n \rightarrow \text{M}(r,\C)$ is a kernel. If $K$ is quasi-invariant with respect to $J,$ then we have 
\begin{equation*}
\mathcal{K}(\bz) = \left( Dg(\bz)^t \otimes (J(\tilde{g}, \bz)^*)^{-1}\right) \mathcal{K}(g(\bz)) \left( \overline{Dg(\bz)} \otimes J(\tilde{g}, \bz)^* \right)
\end{equation*}
for $\tilde{g} \in \tilde{G}$ with $g=p(\tilde{g})$, $p:\tilde{G}\ra G$ is the universal covering map and $\,\bz, \bw \in \mathbb{D}^n.$
\end{prop}

\begin{proof}
Let $\tilde{g} \in \tilde{G}$ with $g=p(\tilde{g})$ and consider the kernel $K_{g} : \mathbb{D}^n \times \mathbb{D}^n \rightarrow \text{M}(r,\C)$ defined by $K_{g}(\bz, \bw) = K(g( \bz), g( \bw)).$
Since $K$ is quasi-invariant with respect to $J$, we have 
\begin{equation*}
K(\bz, \bw) = J(\tilde{g}, \bz) K(g(\bz), g(\bw)) J(\tilde{g}, \bw)^*
\end{equation*}
from which it follows that
\begin{flalign*}
\mathcal{K}^{i j}_{g}(\bz) &= \partial_{i} \left[ K_{g}(\bz, \bz)^{-1} \overline{\partial_{j}} K_{g}(\bz, \bz) \right]\\
&= \partial_i \left[ J(\tilde{g}, \bz)^* K(\bz, \bz)^{-1} J(\tilde{g}, \bz) \left\lbrace J(\tilde{g}, \bz)^{-1} \overline{\partial_{j}} K(\bz, \bz) \left(J(\tilde{g}, \bz)^* \right)^{-1}\right.\right.\\
& ~~~~~~~~~~~~~~+\left.\left. J(\tilde{g}, \bz)^{-1} K(\bz, \bz) \overline{\partial_{j}}\left(J(\tilde{g}, \bz)^* \right)^{-1} \right\rbrace \right]\\
&= \partial_i \left[ J(\tilde{g}, \bz)^* K(\bz, \bz)^{-1} \overline{\partial_{j}} K(\bz, \bz) \left(J(\tilde{g}, \bz)^*\right)^{-1} + J(\tilde{g}, \bz)^* \overline{\partial_{j}} \left(J(\tilde{g}, \bz)^*\right)^{-1} \right]\\
&= J(\tilde{g}, \bz)^* \partial_{i} \left[ K(\bz, \bz)^{-1} \overline{\partial_{j}} K(\bz, \bz) \right] \left(J(\tilde{g}, \bz)^*\right)^{-1}\\
&= J(\tilde{g}, \bz)^* \mathcal{K}^{i j}(\bz) \left(J(\tilde{g}, \bz)^*\right)^{-1}.
\end{flalign*}
This gives us 
\begin{equation}\label{eqn:5.0.1}
\mathcal{K}_g(\bz) = \left( I \otimes J(\tilde{g}, \bz)^*\right) \mathcal{K}(z) \left( I \otimes (J(\tilde{g}, \bz)^*)^{-1} \right).
\end{equation}
Also using the chain rule, we obtain  
\begin{equation}\label{eqn:5.0.2}
\mathcal{K}_{g}(\bz) = \left( Dg(\bz)^t \otimes I\right) \mathcal{K}(g(\bz)) \left( \overline{Dg(\bz)} \otimes I \right).
\end{equation}
Combining \eqref{eqn:5.0.1} and \eqref{eqn:5.0.2}, we have 
\begin{equation*}
\mathcal{K}(\bz) = \left( Dg(\bz)^t \otimes (J(\tilde{g}, \bz)^*)^{-1}\right) \mathcal{K}(g(\bz)) \left( \overline{Dg(\bz)} \otimes J(\tilde{g}, \bz)^* \right).
\end{equation*}
verifying the transformation rule for the curvature $\mathcal K$. 
\end{proof}
The authors have benefited from the discussions with Kui Ji, Dinesh Kumar Keshari and Surjit Kumar in obtaining the curvature formula in Proposition \ref{Transformation rule for curvature}.

Since $G$ acts transitively on $\mathbb D^n,$ there is a $g_{\bz}$  in $G$ with  $g_{\bz}(\bz)=0$ for each $\bz\in\D^n$. Substituting $g_{\bz}$ in the transformation rule for the curvature obtained in Proposition \ref{Transformation rule for curvature}, we see that the curvature at any $\bz \in \mathbb D^n,$ is determined from its value at $0.$

\begin{cor}\label{Determintion of curvature knowing at 0}
With notations and assumptions as in Proposition \ref{Transformation rule for curvature}, we have 
\begin{equation*}
\mathcal{K}(\bz) = \left( Dg_{\bz}(\bz)^t \otimes (J(\tilde{g}_{\bz}, \bz)^*)^{-1}\right) \mathcal{K}(0) \left( \overline{Dg_{\bz}(\bz)} \otimes J(\tilde{g}_{\bz}, \bz)^* \right),
\end{equation*}
where $g_{\bz}$ is in $G$ with  $g_{\bz}(\bz)=0$ and $\tilde{g}_{\bz}\in \tilde{G}$ is such that $p(\tilde{g}_{\bz})=g_{\bz}$ for each $\bz\in\D^n$. 
\end{cor}

\begin{lem}\label{curvature at 0 is diagonal}
Let $J_{\tilde{g}} : \mathbb{D}^n \rightarrow \text{GL}(r, \mathbb{C}),$ $\tilde{g}\in \tilde{G}$ be holomorphic  and $K : \mathbb{D}^n \times \mathbb{D}^n \rightarrow \text{M}(r,\C)$ be a kernel. If $K$ is quasi-invariant with respect to $J$ and $G = {\mob}^n$, then  $\mathcal{K}^{ij}(0)$ is nilpotent whenever $i \neq j$. Furthermore, suppose  $G$ is the group $\Aut(\mathbb{D}^n)$. Then 
$\mathcal{K}^{ij}(0)$  is nilpotent for $1\leq i,j \leq n$, $i\not = j$, and all the $\mathcal{K}^{ii}(0)$ are mutually similar for $1\leq i\leq n$.
\end{lem}

\begin{proof}
Since $K$ is quasi-invariant with respect to $J$, it follows from Lemma \ref{Transformation rule for curvature}  that
$$\mathcal{K}(\bz) = \left( Dg(\bz)^t \otimes (J(\tilde{g}, \bz)^*)^{-1}\right) \mathcal{K}(g(\bz)) \left( \overline{Dg(\bz)} \otimes J(\tilde{g}, \bz)^* \right)$$
for all $g$ in M\"{o}b$^n$ and $\bz$ in $\mathbb{D}^n$.
Let $k \in \mob^n$ be such that $k(z_1, z_2,\ldots, z_n) = (k_1z_1$, $k_2z_2$,\ldots, $k_nz_n)$ for $(z_1, z_2,\ldots, z_n)$ in $\mathbb{D}^n$ where modulus of each $k_i$ is $1$. Then $Dk(0) = \text{diag}(k_1, k_2,\ldots, k_n)$ is the diagonal matrix with diagonal entries $k_1,\hdots,k_n$. Now replacing $g$ by $k$ with $\tilde{k}\in\tilde{G}$ such that $p(\tilde{k})=k$ and $\bz$ by $0$ in the equation appearing above, we get
$$\mathcal{K}(0) = \left( Dk(0)^t \otimes (J(\tilde{k}, 0)^*)^{-1}\right) \mathcal{K}(0) \left( \overbar{Dk(0)} \otimes J(\tilde{k}, 0)^* \right)$$
which is equivalent to the equation:
$$\left( \overbar{Dk(0)} \otimes J(\tilde{k}, 0)^*\right)\mathcal{K}(0) =  \mathcal{K}(0) \left( \overbar{Dk(0)} \otimes J(\tilde{k}, 0)^* \right).$$
Now equating the $(i, j)$th block from both sides, we get 
$$\bar{k}_iJ(\tilde{k}, 0)^* \mathcal{K}^{ij}(0) = \bar{k}_j\mathcal{K}^{ij}(0)J(\tilde{k}, 0)^* .$$
Thus if $i\neq j,$ $ \mathcal{K}^{ij}(0)$ is similar to $ k_i \bar{k}_j\mathcal{K}^{ij}(0)$ for all $k_i,k_j$ in the unit circle. This means that $0$ is the only eigenvalue of $\mathcal{K}^{ij}(0)$, $i \neq j$ and therefore $\mathcal{K}^{ij}(0)$, $i \neq j$, is nilpotent. 

Now assume that $G = \Aut(\mathbb{D}^n)$. Since $K$ is quasi-invariant with respect to $J$, $K$ is also quasi-invariant with respect to $J|_{\widetilde{\mob}^n \times \mathbb{D}^n}$. It then follows from the first part that $\mathcal{K}^{ij}$ is nilpotent if $i \neq j$.

Let $\sigma_k \in \Aut (\mathbb{D}^n)$ be the automorphism such that $\sigma_k (z_1, z_2,\ldots, z_n) = (k_2z_2, k_1z_1,\ldots, k_nz_n)$ for $(z_1, z_2,\ldots, z_n)$ in $\mathbb{D}^n$ where each $k_i$ is in the unit circle. Then $$D\sigma_k (0) = \begin{pmatrix}
    0       & k_2 & 0 & \dots & 0 \\
    k_1       & 0 & 0 & \dots & 0 \\
    \vdots & \vdots & \vdots & \ddots & \vdots \\
    0       & 0 & 0 & \dots & k_n
\end{pmatrix}.$$
Replacing $g$ by $\sigma_k$ and $\bz$ by $0$ in Lemma \ref{Transformation rule for curvature}, we obtain  
$$\overbar{D\sigma_k (0)} \otimes J(\tilde{\sigma_k}, 0)^* \mathcal{K}(0) = \mathcal{K}(0) \overbar{D\sigma_k (0)} \otimes J(\tilde{\sigma_k}, 0)^*.$$
Equating the $(1,2)$ block from  both sides of the equation, we get 
$$J(\tilde{\sigma_k}, 0)^* \mathcal{K}^{22}(0) = \mathcal{K}^{11}(0) J(\tilde{\sigma_k}, 0)^* .$$
Since $J(\tilde{\sigma_k}, 0)^*$ is invertible, it follows that $\mathcal{K}^{11}(0)$ and $\mathcal{K}^{22}(0)$ are similar.
Similar reasoning shows that $\mathcal{K}^{ii}(0)$ and $\mathcal{K}^{i+1\,\,i+1}(0)$ are similar for all $i$.
\end{proof}

\section{Homogeneous tuples in $\mathrm B_1(\mathbb D^n)$}\label{Section 3}
In this section, we describe all homogeneous operators in $\mathrm B_1(\mathbb{D}^n)$ with respect to both the group $\mob^n$ and the full automorphism group $ \Aut(\mathbb D^n).$ 
\begin{thm}\label{rank one charaterization}
Assume that the adjoint of the $n$-tuple of multiplication operators $(M_{z_1}, M_{z_2}, \ldots,$ $ M_{z_n}),$ defined on a reproducing kernel Hilbert space $H_K,$ is in $\mathrm B_1(\mathbb{D}^n).$  Then 

\noindent(a) the $n$ - tuple $(M_{z_1}, M_{z_2}, \ldots, M_{z_n})$  is homogeneous with respect to $\mob^n$ if and only if 
$$K(\bz,\bw) = h(\bz) \Big (\prod_{i = 1}^{n} \frac{1}{(1 - z_{i}\overline{w}_{i})^{\lambda_i}}\Big ) \overline{h(\bw)},\,\, \bz,\bw\in \mathbb D^n,\,\, \lambda_i> 0,$$  for some non-vanishing holomorphic function $h:\mathbb D^n \to \mathbb C;$

\noindent(b) the $n$ - tuple $(M_{z_1}, M_{z_2}, \ldots, M_{z_n})$  is homogeneous with respect to  $Aut(\mathbb{D}^n)$ if and only if 
$$K(\bz,\bw) = h(\bz) \Big ( \prod_{i = 1}^{n} \frac{1}{(1 - z_{i}\overline{w}_{i})^{\lambda}} \Big ) \overline{h(\bw)},\,\, \bz,\bw\in \mathbb D^n,\,\, 
\lambda> 0,$$ for some non-vanishing holomorphic function $h:\mathbb D^n \to \mathbb C.$
\end{thm}

\begin{proof}
(a) It is well-known that the $n$ - tuple $(M_{z_1}^*, \ldots, M_{z_n}^*)$ 
on the Hilbert space $H_K$ with $K(\bz,\bw) = \prod_{i = 1}^{n}(1 - z_{i}\overline{w_{i}})^{-\lambda_i}$ is in $\mathrm B_1(\mathbb D^n).$ It is also easy to verify that this reproducing kernel is quasi-invariant \w $J_{\tilde{g}}(\bz)=\prod_{i=1}^n(g_i'(z_i))^{\frac{\l_i}{2}}$, for $\bz\in \D^n$ and $g=(g_1,\hdots,g_n)$ with $g_i(z)=e^{i\theta_i}\frac{z_i-a_i}{1-\ov{a}_iz_i}$, implying that the $n$ - tuple $(M_{z_1}^*, \ldots, M_{z_n}^*)$ is homogeneous with respect to M\"{o}b$^n.$ This is the proof in one direction. 

For the proof in the other direction, note that  the existence of a holomorphic map $J_{\tilde{g}},$ $\tilde{g} \in \widetilde{\mob}^n,$ with $p(\tilde{g})=g$ such that 
$$K(\bz, \bw) = J(\tilde{g}, \bz) K(g \bz,g \bw) \overline{J(\tilde{g}, \bw)},\,\,\bz, \bw \in \mathbb{D}^n,\,\,\tilde{g} \in \widetilde{\mob}^n$$ 
follows from \cite[Theorem 3.1]{OICHO}. 

Let $a_{ij}$ be the $i j$ - th entry of $\mathcal{K}(0)$. From Lemma \ref{curvature at 0 is diagonal}, it follows that $a_{ij} = 0$ if $i \neq j$. This shows that $\mathcal{K}(0) = \text{diag}(a_{11}, a_{22},\ldots,a_{nn})$.
Now Corollary \ref{Determintion of curvature knowing at 0} gives
$$\mathcal{K}(\bz) = Dg_{\bz} (\bz)^t \mathcal{K}(0) \overline{Dg_{\bz} (\bz)} = \text{diag} \left(\frac{a_{11}}{(1 - |z_1|^2)^2}, \frac{a_{22}}{(1 - |z_2 |^2)^2},\ldots, \frac{a_{nn}}{(1 - |z_n |^2)^2}\right),$$
where $g_{\bz}(\bw) = - (\frac{w_1-z_1}{1-\bar{z_1}w_1},\hdots,\frac{w_n-z_n}{1-\bar{z_n}w_n}),\,\,\bz,\bw \in \mathbb{D}^n.$
Let $ \lambda_i = a_{ii},$ $1\leq i \leq n$. Recalling that $\mathcal K_1 = \mathcal K_2$ if and only if $K_2 = h K_1 \bar{h}$ for some non-vanishing holomorphic function $h,$ we conclude 
that $$K(\bz, \bw) = h(\bz) \left( \prod_{i = 1}^{n} \frac{1}{(1 - z_{i}\overline{w_{i}})^{\lambda_i}}\right) \overline{h(\bw)},$$ $h$ non-vanishing holomorphic on $\mathbb D^n.$ Since $K$ is a positive definite kernel, it follows that $\lambda_i > 0,$ $1 \leq i \leq n.$

(b) The proof in the forward direction follows from the proof in the same direction of part (a). For the other direction, note that $K$ is quasi-invariant \w  $\aut(\mathbb{D}^n)$. So Proposition \ref{Transformation rule for curvature} yields  
\begin{equation}\label{eqn:5.1.x}
\mathcal{K}(\bz) = Dg(\bz)^t  \mathcal{K}(g(\bz))  \overline{Dg(\bz)},\,\,\bz \in \mathbb{D}^n\,\,g \in \aut(\mathbb{D}^n).
\end{equation}

Since $\mob^n$ is a subgroup of $\aut(\mathbb{D}^n)$, it follows that $(M_{z_1}, M_{z_2}, \ldots, M_{z_n})$ is homogeneous with respect to the group $\mob^n$.
Therefore, $\mathcal{K}(0) = \text{diag}(a_{11}, a_{22},\ldots,a_{nn})$ where $a_{ii} > 0,$ $1 \leq i \leq n$.

Now applying Lemma \ref{curvature at 0 is diagonal}, we obtain  $a_{ii} = a_{11}$ for all $i$. Putting $\lambda = a_{11},$  we have $\mathcal{K}(0) = \text{diag}(\lambda, \lambda,\ldots, \lambda)$.
Now Corollary \ref{Determintion of curvature knowing at 0} gives 
$$\mathcal{K}(\bz) = Dg_{\bz} (\bz)^t \mathcal{K}(0) \overline{Dg_{\bz} (\bz)} = \text{diag} \left(\frac{\lambda}{(1 - |z_1|^2)^2}, \frac{\lambda}{(1 - |z_2 |^2)^2},\ldots, \frac{\lambda}{(1 - |z_n |^2)^2} \right),\,\,z \in \mathbb{D}^n,$$
where $g_{\bz}(\bw) = - (\frac{w_1-z_1}{1-\bar{z_1}w_1},\hdots,\frac{w_n-z_n}{1-\bar{z_n}w_n}),\,\,\bz,\bw \in \mathbb{D}^n.$ This implies that $$K(\bz, \bw) = h(\bz) \left( \prod_{i = 1}^{n} \frac{1}{(1 - z_{i}\overline{w_{i}})^{\lambda}} \right) \overline{h(\bw)}$$ for some  holomorphic function $h$ on $\mathbb D^n.$ 
\end{proof}

\section{Homogeneous vector bundles}\label{Section 4}

In this section, we study {\h}s over $\D^n$ homogeneous \w some closed subgroup of the group of bi-holomorphic automorphisms of $\D^n$. We begin with a commuting tuple which is homogeneous with respect to the group $\mob^n.$ 



Let $\T=(T_1,\hdots,T_n)\in \mathrm B_r(\D^n)$. It follows that $\T$ is unitarily equivalent to $g(\T)$ if and only if $g$ lifts to a bundle automorphism of the {\h} $E$ associated to $\T$ for $g\in\mob^n$. Further, we recall from \cite[Theorem 6.1]{HOHS} that the universal covering group $\widetilde{\mob}^n$ of $\mob$ acts on $E$ as described in the theorem below.

\begin{thm}\label{grpactvec}
Let $\T=(T_1,\hdots,T_n)$ be a tuple of commuting operators in $\mathrm B_r(\D^n)$ homogeneous \w the group $\mob^n$. Then the universal covering group $\widetilde{\mob}^n$ of $\mob^n$ acts on the {\h} $E$ associated to $\T$.
\end{thm}

On the other hand, we point out from \cite[Theorem 3.1]{OICHO} that $\T=(T_1,\hdots,T_n)\in \mathrm B_r(\D^n)$ is homogeneous \w $\mob^n$ if and only if the tuple of multiplication operators $\M=(M_{z_1},\hdots,M_{z_n})$ on $\H_K$ is homogeneous \w $\mob^n$ which is again equivalent to the fact that $K$ satisfies the following equation 
\beq\label{qinv}
K(\bz,\bw) &=& J(\tilde{g},\bz)K(g(\bz),g(\bw))J(\tilde{g},w)^*,~~~\bz,\bw\in\D^n,~\tilde{g}\in\widetilde{\mob}^n,
\eeq
where $g=p(\tilde{g})$, $p:\widetilde{\mob}^n\ra \mob^n$ is the universal covering map and, for $\tilde{g}\in\widetilde{\mob}^n$, $J({\tilde{g}},\cdot):\D^n\ra \text{GL}(r,\C)$ is a holomorphic mapping. Moreover, it is observed from the discussion made in \cite[Section 1.4]{ACHOCD} that for a homogeneous $n$ - tuple of operators $\T=(T_1,\hdots,T_n)$, the mapping $J({\tilde{g}},\cdot)$ is obtained as follows:
\beq\label{cocycle}
J({\tilde{g}},\bz) &=& \phi_{\bz}\circ \tilde{g}_{\bz}^{-1}\circ \phi_{g(\bz)}^{-1}
\eeq
where $\phi$ is global trivialization of the bundle $E\ra \D^n$, $\phi_{\bz}:E_{\bz}\ra\C^r$, $\phi_{g(\bz)}:E_{g(\bz)}\ra\C^r$ and $\tilde{g}_{\bz}:E_{\bz}\ra E_{g(\bz)}$ are all linear isomorphisms. We observe, for $\tilde{g},\tilde{h}\in\widetilde{\mob}^n$ with $g=p(\tilde{g})$, $h=p(\tilde{h})$ and $\bz\in\D^n$, that 
$$J(\tilde{h}\tilde{g},\bz)=J(\tilde{g},\bz)J(\tilde{h},g(\bz))$$ which implies that $J:\widetilde{\mob}^n\times\D^n\ra\text{GL}(r,\C)$ is a cocycle. Note that any other choice of trivialization of $E$ yields an equivalent cocycle.

Let $g_{\bz}$ be an element in $\mob^n$ which maps $\bz$ to $0$, that is, $g_{\bz}(\bz)=0$ and $\tilde{g}_{\bz}\in\widetilde{\mob}^n$ be such that $p(\tilde{g}_{\bz})=g_{\bz}$. For a quasi-invariant kernel $K$, we have that
\beq\label{cocyclevsker}
K(\bz,\bz) &=& J({\tilde{g}_{\bz}},\bz)K(0,0)J({\tilde{g}_{\bz}},\bz)^*.
\eeq
Thus it shows that $K(\bz,\bz)$ is uniquely determined by $K(0,0)$ provided $J:\widetilde{\mob}^n\times \D^n\ra \text{GL}(r,\C)$ is known. Let $\K$ be the subgroup of $\mob^n$ consisting of elements which fix the origin and $\tilde{\K}\subset \widetilde{\mob}^n$ be the subgroup such that $p(\tilde{\K})=\K$. Then, for any $\tilde{k}\in\tilde{\K}$, we have that 
\beq\label{K(0,0)}
K(0,0)= J({\tilde{k}},0)K(0,0)J({\tilde{k}},0)^*.
\eeq
In other words, the inner product $\<K(0,0)\cdot,\cdot\>$ is invariant under $J({\tilde{k}},0)$ for every $\tilde{k}\in\tilde{\K}$. Thus any positive definite matrix $K(0,0)$ defines, via \eqref{cocyclevsker}, a hermitian structure on the homogeneous vector bundle $E\ra\D^n$ determined by $J({\tilde{g}},\bz)$, $\bz\in\D^n$, $\tilde{g}\in\widetilde{\mob}^n$. Therefore, finding all homogeneous operators in $\mathrm B_r(\D^n)$ amounts to
\begin{itemize}
\item[(i)] obtaining a classification of all homogeneous {\h}s over $\D^n$ by determining all the cocycles $J:\widetilde{\mob}^n\times\D^n\ra\text{GL}(r,\C)$ holomorphic in $z$;
\item[(ii)] determining among these which homogeneous bundles ``correspond'' to a holomorphic curve in the Grassmannian of rank $r$ of some Hilbert space $\mathcal H$ possessing a reproducing kernel $K$
\item[(iii)] and finally showing that the adjoint of the multiplication by the coordinate functions on the Hilbert space is bounded and is in the class $\mathrm B_r(\D^n)$.
\end{itemize}
The bulk of the work in this paper goes into settling item (i) of the list we have just made up. This involves understanding the representations of the group $\widetilde{\mob}^n$ given the detailed knowledge of $\widetilde{\mob}$. To settle item (ii) of this list,  we have to find all $K(0,0)$ satisfying \eqref{K(0,0)} such that the polarization of the equation \eqref{cocyclevsker} yields a positive definite kernel $K$. Starting from the  positive definite kernel $K$,  there is a canonical construction producing a Hilbert space $\mathcal H_K$ for which $K$ serves as the reproducing kernel, (see \cite[Theorem 2.14]{AITRKHS}). To settle (iii), one has to determine if the commuting $n$ - tuple of multiplication by the coordinate functions on the Hilbert space $\mathcal H_K$ are bounded and when they are, if the adjoint of this $d$ - tuple of operators belongs to  $\mathrm B_r(\D^n)$. Although this construction of the cocycle associated to a homogeneous vector bundle is known in literature (cf. \cite[Section 1.5, Section 2]{ACHOCD}), a sketch of it has been presented here for the sake of completeness.

Recall that  $\text{SU}(1, 1)$ is the $2$ - fold covering group of $\mob$ and let $\widetilde{\mob}$ be the universal covering group of $\text{SU}(1, 1).$  The unit disc $\mathbb{D}$ admits an action of the group $\text{SU}(1, 1)$ by the rule, 
$$g(z) = \frac{az + b}{\bar{b}z + \bar{a}}, \,\,g = \left(\begin{matrix}
    a & b\\
    \bar{b} & \bar{a}
\end{matrix}\right) \in \text{SU}(1, 1), z \in \mathbb{D}.$$ Therefore, the group $G:=\text{SU}(1,1)^n$ acts on $\mathbb{D}^n.$ Let $\tilde{G}$ be the universal covering group of  $G$. Evidently, composing with the covering map, an action of $\tilde{G}$ on $\mathbb{D}^n$ is obtained. Thus, $\D^n\cong \tilde{G}/\tilde{\K}\cong G/\K$. We now describe the construction of all cocycles giving rise to an irreducible {\h} over $\D^n$ which is homogeneous \w the group $\tilde{G}$. 

In the discussion below, we follow the notation of \cite{ACHOCD}. The Lie algebra $\mathfrak{su}(1,1)$ of $\text{SU}(1,1)$ is spanned by
$$X_1 = \frac{1}{2}\begin{pmatrix}
0 & 1\\
1 & 0
\end{pmatrix},\, X_0 = \frac{1}{2}\begin{pmatrix}
i & 0\\
0 & -i
\end{pmatrix}\,\,\mbox{and}\,\,Y = \frac{1}{2}\begin{pmatrix}
0 & i\\
-i & 0
\end{pmatrix}.$$ Then the subalgebra corresponding to the rotation subgroup of SU$(1,1)$ is spanned by $X_0$.

Let $\mathfrak{su}(1,1)^{\mathbb{C}}$ be the complexification of $\mathfrak{su}(1,1)$. Then $\mathfrak{su}(1,1)^{\mathbb{C}}$ is the Lie algebra of the complexification of the group SU$(1, 1)$ which is SL$(2, \mathbb{C})$. The Lie algebra $\mathfrak{su}(1,1)^{\mathbb{C}}$ is spanned by 
$$h = -iX_0 = \frac{1}{2}\begin{pmatrix}
1 & 0\\
0 & -1
\end{pmatrix},\, x = X_1 + iY = \begin{pmatrix}
0 & 1\\
0 & 0
\end{pmatrix}\,\,\mbox{and}\,\,y = X_1 - iY = \begin{pmatrix}
0 & 0\\
1 & 0
\end{pmatrix}.$$
Let $K^{\mathbb{C}} = \left\lbrace\left( \begin{smallmatrix}
z & 0\\
0 & z^{-1}
\end{smallmatrix}\right) : z \in \mathbb{C} \setminus \{0\} \right\rbrace$, $P^+ = \left\lbrace \left(\begin{smallmatrix}
1 & z\\
0 & 1
\end{smallmatrix}\right) : z \in \mathbb{C} \right\rbrace$ and $P^- = \left\lbrace \left(\begin{smallmatrix}
1 & 0\\
z & 1
\end{smallmatrix}\right) : z \in \mathbb{C} \right\rbrace$ be the subgroups of SL$(2,\C)$. The Lie algebras of $K^{\mathbb{C}}$, $P^+$ and $P^-$ are $\mathfrak{t}^{\mathbb{C}} = \left\lbrace \left(\begin{smallmatrix}
c & 0\\
0 & -c
\end{smallmatrix}\right) : c \in \mathbb{C} \right\rbrace$, $\mathfrak{p}^{+} = \left\lbrace \left(\begin{smallmatrix}
0 & c\\
0 & 0
\end{smallmatrix}\right) : c \in \mathbb{C} \right\rbrace$ and $\mathfrak{p}^{-} = \left\lbrace \left(\begin{smallmatrix}
0 & 0\\
c & 0
\end{smallmatrix}\right) : c \in \mathbb{C} \right\rbrace$, respectively. Let $\mathfrak{b}$ denote the Lie algebra spanned by $\{h, y\}$. Then $\mathfrak{b}$ is the Lie algebra for the group $K^{\mathbb{C}}P^-$ which is a closed subgroup of SL$(2, \mathbb{C})$. It follows that the Lie algebra $\mathfrak{g}$ of $G$ (which is also the Lie algebra of $\tilde{G}$) is the direct sum of $n$ copies of $\mathfrak{su}(1,1)$ and Lie algebra $\mathfrak{k}$ of the subgroup $\K$ is the direct sum of $n$ copies of the sub-algebra spanned by $X_0$. Let $\mathfrak{g}^{\C}$ and $\mathfrak{k}^{\C}$ be the complexified Lie algebras. We also let $G^{\C}$ be complex Lie group corresponding to the Lie algebra $\mathfrak{g}^{\C}$ and $\B$ be the complex Lie subgroup $(K^{\C}P^{-})^n$ corresponding to the sub-algebra $\mathfrak{b}^n := \mathfrak{b}\oplus\cdots\oplus\mathfrak{b}$, $n$ times. Note that $G^{\C}$ is simply connected. We observe that 
\beq\label{incl}
\D^n\cong\tilde{G}/\tilde{\K}\hookrightarrow G^{\C}/\B\cong (\S^2)^n
\eeq as an open subset. Moreover, $G^{\C}\ra G^{\C}/\B$ is a holomorphic principal $\B$ - bundle over $(\S^2)^n$.

\begin{thm}\label{construction of cocycle}
Let $E$ be an irreducible {\h} over $\D^n$ which is homogeneous \w the group $\tilde{G}$. Then the cocycle associated to $E\ra\D^n$ takes the form 
\beq\label{cocyclecomp2}
J(\tilde{g},\bz) = \prod_{i=1}^n(g_i'(z_i))^{\a_i} \rho^0(s(\bz)^{-1}g^{-1}s(g(\bz))),
\eeq
where $s(\bz)$ is a local holomorphic section of the principal $\B$ - bundle $G^{\C}\ra(\S)^n$ over the open set $\D^n\hookrightarrow(\S^2)^n$, $\rho^0$ is a representation of the Lie group $\B$, $(g_1,\hdots,g_n)=g=p(\tilde{g})$ and $p:\tilde{G}\ra G$ is the universal covering mapping.
\end{thm}
\begin{proof}
Let $E(\tilde{G},\rho)\ra\D^n$ denote the smooth vector bundle obtained from the principal $\tilde{\K}$ - bundle $\tilde{G}\ra\tilde{G}/\tilde{\K}$ and the representation $\rho:\tilde{\K} \to \text{GL}(r, \mathbb C)$ of the closed subgroup $\tilde{\K}$. Since the vector bundle $E$ is homogeneous \w $\tilde{G}$, it can be shown that there exists a $r$ dimensional representation $\rho$ of $\tilde{\K}$ such that $E$ is of the form $E(\tilde{G},\rho)$ equipped with an isometric action of $\tilde{G}$. Moreover, since the hermitian structure on $E\ra\D^n$ is $\tilde{G}$ - invariant, there is a $\rho(\tilde{\K})$ - invariant inner product on the representation space $\C^r$.

Since the representation space $\C^r$ admits a $\rho(\tilde{\K})$ - invariant inner product and since $\tilde{\K}$ is an abelian group, for all $\tilde{k}\in\tilde{\K}$, $\rho(\tilde{k})$ are simultaneously diagonalizable. This representation $\rho$ induces a representation of the Lie algebra $\mathfrak{k}$ which can be extended to obtain a representation of $\mathfrak{k}^{\C}$. Let us also denote this new representation by the same letter $\rho$. It follows that $\rho(h_i)$ is a diagonal matrix \w some suitable basis of $\C^r$ where $\{h_i=(0,\hdots,h,\hdots,0),y_i=(0,\hdots,y,\hdots,0)\}$ is the basis of $i$-th subalgebra $\mathfrak{b}_i$ of $\mathfrak{b}^n$. Furthermore, since $E\ra \D^n$ is a holomorphic vector bundle it follows from \cite[Theorem 3.6]{HHVB} that $\rho$ can be extended to obtain a representation of the Lie subalgebra $\mathfrak{b}^n$. By a slight abuse of notation, we denote this new representation by the same letter $\rho$.

Since the  {\h} $E\ra\D^n$ is irreducible by assumption, it follows that the representation $\rho$ is indecomposable. Consequently, the identities $[\rho(h_i),\rho(h_j)]=[\rho(h_i),\rho(y_j)]=[\rho(y_i),\rho(y_j)]=0$ for $i\neq j$, $i,j=1,\hdots,n$, force each $\rho(h_i)$ to have an uninterrupted string of eigenvalues  $-\a_i,\hdots,-(\a_i+k_i-1)$ with appropriate multiplicities where $-\a_i$ and $-(\a_i+k_i-1)$ are the largest and the smallest eigenvalues of $\rho(h_i)$, respectively, for some positive integers $k_i$ and $1\leq i\leq n$. However, the case $n = 2$ which we specifically need, is recorded in Lemma \ref{properties of indecomposable rep} (ii). Note that $\rho$ can be written as the tensor product of the one dimensional representation $\si$ given by $\si(h_i)=\a_i$ and $\si(y_i)=0$,  and the representation $\rho^0$ given by $\rho^0(h_i)=\rho(h_i)+\a_iI$, $\rho^0(y_i)=\rho(y_i)$, $1\leq i\leq n$, where $y_i=(0,\hdots,y,\hdots,0)$ with $y$ in the $i$-th position of the tuple. 

Note that $\rho^0$ gives rise to a holomorphic representation, 
 to be denoted by the same letter $\rho^0$, of $\mathbb{B}$ whose derivative at identity is the representation $\rho^0$. Indeed, writing each element $\left( \begin{smallmatrix}
a & 0\\
b & a^{-1}
\end{smallmatrix}\right)\in K^{\C}P^-$ with $a\in\C\setminus\{0\}$ and $b\in\C$ as $\exp(ba^{-1}y+2(\log a)h)$, the map 
\begin{equation}\label{eqn xx}
\begin{pmatrix}
a & 0\\
b & a^{-1}
\end{pmatrix}\mapsto \exp\left(\frac{b}{a}\rho^0(y)\right)\exp(2(\log a)\rho^0(h))
\end{equation}
defines a representation of $K^{\C}P^{-}$. We then extend this representation component-wise to obtain a representation of $\mathbb{B}$. Thus $\rho^0$ together with the holomorphic principal $\B$ - bundle $G^{\C}\ra (\S^2)^n$ give rise to the holomorphic vector bundle $E(G^{\C},\rho^0)\ra (\S^2)^n$ which is homogeneous \w the group $G^{\C}$. On the other hand, the one dimensional representation $\si$ also gives rise to a holomorphic line bundle $L(G^{\C},\si)\ra(\S^2)^n$. Moreover, the holomorphic vector bundle $E\ra\D^n$ is isomorphic to the holomorphic bundle $E(G^{\C},\si)|_{\D^n}\otimes E(G^{\C},\rho^0)|_{\D^n}\ra \D^n$ obtained by pulling back the bundle $E(G^{\C},\si)\otimes E(G^{\C},\rho^0)\ra (\S^2)^n$ via the inclusion $\D^n\hookrightarrow (\S^2)^n$ mentioned in \eqref{incl}.

From the discussion in the preceding paragraphs, it is enough to compute the cocycle corresponding to the holomorphic homogeneous vector bundle $E(G^{\C},\rho^0)|_{\D^n}\ra \D^n$. For this, we take a local holomorphic section $s(\bz)$ of $G^{\C}\ra (\S^2)^n$ over the open set $\D^n\hookrightarrow(\S^2)^n$ which induces a trivialization of the vector bundle $E(G^{\C},\rho^0)|_{\D^n}\ra \D^n$. Then applying \eqref{cocycle} for this trivialization we have that the cocycle corresponding to the holomorphic homogeneous vector bundle $E(G^{\C},\rho^0)|_{\D^n}\ra \D^n$ is 
\beq\label{cocyclecomp1}
J^0(\tilde{g},\bz) &=& \rho^0(s(\bz)^{-1}g^{-1}s(g(\bz)))
\eeq
where $g=p(\tilde{g})$, $p:\tilde{G}\ra \text{SU}(1,1)^n$ is the universal covering map. Thus the proof follows from the fact that the cocycle corresponding to the line bundle $E(G^{\C},\si)|_{\D^n}\ra \D^n$ is $\prod_{i=1}^n(g_i'(z_i))^{\a_i}$. \end{proof}

Now in the following corollary, an explicit formula for the cocycle $J$ with respect to a suitable local section of $G^{\C}\ra (\S^2)^n$ is provided where it is observed that we actually need an irreducible representation of the Lie algebra $\mathfrak{b}$ of $\B$ to construct $J$ in our case. 

\begin{cor}
Let  $g = \left(\!\!\left(\begin{smallmatrix}
a_i & b_i\\
c_i & d_i
\end{smallmatrix} \right)\!\!\right)_{i=1}^{n} \in \text{SU}(1,1)^n$ and take the local holomorphic section $s(\bz)$ of $G^{\C}\ra (\S^2)^n$ over the open set $\D^n\hookrightarrow(\S^2)^n$ defined by $s(\bz)=\left(\!\left(\begin{smallmatrix}1 & z_1\\ 0 & 1\end{smallmatrix}\right),\left(\begin{smallmatrix}1 & z_2\\ 0 & 1\end{smallmatrix}\right),\hdots,\left(\begin{smallmatrix}1 & z_n\\ 0 & 1\end{smallmatrix}\right)\!\right)$. Then the cocycle $J$ associated to the holomorphic homogeneous vector bundle $E\ra \D^n$ turns out to be 
\beq\label{cocyclecomp3}
\nonumber J(\tilde{g},\bz) &=&\Big ( \prod_{i=1}^n(g_i'(z_i))^{\a_i}\Big )\exp\left(\!\!\left(\!\frac{-c_1}{c_1z_1 + d_1}\!\right) \rho^0(y_1)\!\!\right)\, \exp\left(\left(2 \log(c_1z_1 + d_1)\right)\rho^0(h_1)\right)\times\cdots\\
&\times &  \exp\left(\!\!\left(\!\frac{-c_n}{c_nz_n + d_n}\!\right) \rho^0(y_n)\!\!\right)\,\exp\left(\left(2 \log(c_nz_n + d_n)\right)\rho^0(h_n) \right).
\eeq

\end{cor}

\begin{proof}
Applying \eqref{cocyclecomp1} for the trivialization induced by the holomorphic section $s(\bz)$, we have that the cocycle corresponding to the holomorphic homogeneous vector bundle $E(G^{\C},\rho^0)|_{\D^n}\ra \D^n$ is 
\begin{flalign*}
& J^0(\tilde{g},\bz)
= \rho^0\left(\!\exp\left(\!\!\left(\!\frac{-c_1}{c_1z_1 + d_1}\!\right) y_1\!\!\right) \exp\left(\left(2\log(c_1z_1 + d_1)\right)h_1\right)\, \exp\left(\!\!\left(\!\frac{-c_2}{c_2z_2 + d_2}\!\right) y_2\!\!\right)\! \right.\\
& \left. \exp\left(\left(2 \log(c_2z_2 + d_2)\right)h_2\right)\cdots \exp\left(\!\!\left(\!\frac{-c_n}{c_nz_n + d_n}\!\right) y_n\!\!\right)\, \exp\left(\left(2  \log(c_nz_n + d_n)\right)h_n \right)\right)\\
&= \exp\left(\!\!\left(\!\frac{-c_1}{c_1z_1 + d_1}\!\right) \rho^0(y_1)\!\!\right)\, \exp\left(\left(2 \log(c_1z_1 + d_1)\right)\rho^0(h_1)\right)\, \exp\left(\!\!\left(\!\frac{-c_2}{c_2z_2 + d_2}\!\right) \rho^0(y_2)\!\!\right)\, \\
& \exp\left(\left(2 \log(c_2z_2 + d_2)\right)\rho^0(h_2)\right)\cdots \exp\left(\!\!\left(\!\frac{-c_n}{c_nz_n + d_n}\!\right) \rho^0(y_n)\!\!\right)\, \exp\left(\left(2 \log(c_nz_n + d_n)\right)\rho^0(h_n) \right),
\end{flalign*}
where the last equality holds due to the definition of $\rho^0$ given by the equation \eqref{eqn xx}. Consequently, the proof follows from Theorem \ref{construction of cocycle}.
\end{proof}

Thus every $r$ - cocyle on $\tilde{G} \times \mathbb{D}^n$ is obtained from a $r$ dimensional indecomposable representation of $\mathfrak{b}^n$ such that $\rho$ is diagonalizable on the subalgebra spanned by the set $\{h_1,\hdots,h_n\}$ with $h_i=(0,\hdots,0,h,0,\hdots,0)$, where $h$ is in the $i$-th slot. So in order to characterize all $r$ - cocyles on $\tilde{G} \times \mathbb{D}^n$, it is enough to classify all indecomposable $r$ dimensional representations $\rho$ of $\mathfrak{b}^n$ such that each $\rho(h_i),\,\,1 \leq i \leq n$, is diagonalizable. 

Consequently, in what follows, we assume without loss of generality that each $\rho(h_i),\,\,1 \leq i \leq n$, is diagonalizable  in the representation $\rho$ of $\mathfrak{b}^n$.

The description of all such indecomposable representations of $\mathfrak{b}^2$ occupies the following section.

\section{Classification of irreducible homogeneous {\h}s over $\D^2$}\label{Section 5}
We describe, in this section, all irreducible {\h}s over $\D^2$ homogeneous \w $\mob^2$ whose associated representations are \textit{multiplicity-free}. Note that it is equivalent to classifying all indecomposable \textit{multiplicity-free} representations of the solvable Lie subalgebra $\mathfrak{b}^2$ of $\mathfrak{sl}(2,\C) \oplus \mathfrak{sl}(2,\C)$ such that each $\rho(h_i)$ is diagonalizable as described in the previous section. Recall from Section 4 in \cite{MFHOCD} that a representation $\rho$ of $\mathfrak{b}$ is called multiplicity-free representation if each eigenspace of $\rho(h)$ is one dimensional. Adapting from the one variable case, we define the \textit{multiplicity-free} representations as follows.


\begin{defn}\label{mult free}
A representation $\rho : \mathfrak{b}^n \to \mathfrak{gl}(r, \mathbb{C})$ is said to be multiplicity-free if the linear map $D_{\rho} : \mathbb{C}^r \to \mathbb{C}^r \oplus\cdots\oplus \mathbb{C}^r$, defined by $v \mapsto (\rho(h_1)v,\hdots, \rho(h_n)v)$ has distinct joint eigenvalues, where  $h_i=(0,\hdots,0,h,0,\hdots,0)$,  $h\in\mathfrak{b}$ is in the $i$-th slot, $i=1,2,\hdots,n$. In other words, each subspace which is simultaneously an eigenspace of each one of the $\rho(h_i)$ (this subspace is called joint eigenspace of $D_{\rho}$) is $1$-dimensional.
\end{defn}

Let $V_\theta$ denote the joint eigenspace of $D_{\rho}$ corresponding to the joint eigenvalue $\theta = (\theta_1, \theta_2)$. It follows from the definition of the representation $\rho$ along with the identities
$$[h_i, h_j] = [h_i, y_j] = [y_i, y_j] = 0,\,\,\mbox{if}\,\, i \neq j~\text{and}~[h_i, y_i] = -y_i,\,\, i = 1, 2,$$
that $\rho(y_i)$ maps $V_\theta$ to either $0$ or $V_{\theta - \epsilon_i}$, $i = 1, 2$. We now describe some elementary properties of multiplicity-free representations of $\mathfrak{b}^2$ in the following lemma.

\begin{lem}\label{properties of indecomposable rep}
For a fixed but arbitrary multiplicity-free representation $\rho : \mathfrak{b}^2 \to \mathfrak{gl}(r, \mathbb{C})$, the following two statements hold.

\begin{itemize}
\item[(i)] If $(\lambda, \mu)$, $(\lambda - 1, \mu)$, $(\lambda, \mu - 1)$ and $(\lambda - 1, \mu - 1)$ are all joint eigenvalues of $D_{\rho}$, then one of $\rho(y_1)(V_{(\lambda, \mu)})$ or $\rho(y_2)(V_{(\lambda-1, \mu)})$ is $0$ if and only if $\rho(y_1)(V_{(\lambda, \mu-1)}) = 0$ or $\rho(y_2)(V_{(\lambda, \mu)}) = 0$.

\item[(ii)] For some $i \in \{1, 2\}$, suppose $\alpha, \beta$ are two eigenvalues of $\rho(h_i)$ with $\alpha < \beta$ and no $x\in\mathbb{R}$ with $\alpha < x < \beta$ is an eigenvalue of $\rho(h_i)$. Then $\rho$ is decomposable, provided $\beta - \alpha \neq 1$.
\end{itemize}
\end{lem}

\begin{proof}
(i) It directly follows from the property that $\rho(y_1)$ and $\rho(y_2)$ commute. In other words, the following diagram commutes.
$$
\xymatrix{
V_{(\l,\mu)}\ar[rr]_{\rho(y_1)}^{\circled{1}} \ar[d]_{\circled{4}}^{\rho(y_2)} && V_{(\l-1,\mu)}\ar[d]_{\rho(y_2)}^{\circled{2}}\\
V_{(\l,\mu-1)}\ar[rr]_{\circled{3}}^{\rho(y_1)} && V_{(\l-1,\mu-1)}}
$$
So the statement can be rephrased as follows: absence of one of the arrows \circled{1} and \circled{2} forces at least one of the arrows of \circled{3} and \circled{4} to be absent in the above picture and vice versa.

(ii) Without loss of generality, we assume that $\alpha$ and $\beta$ are eigenvalues of $\rho(h_1)$. It then follows from the identity $[\rho(h_1), \rho(y_1)] = -\rho(y_1)$ that $\rho(y_1)v = 0$ for any eigenvector $v$ of $\rho(h_1)$ corresponding to the eigenvalue $\alpha$. So consider two subspaces $W_1 = \oplus_{t \leq \alpha} W_t$ and $W_2 = \oplus_{t \geq \beta}W_t$ where the sum is taken over all eigenspaces $W_t$ of $\rho(h_1)$ associated to the eigenvalue $t$. We now observe that both $W_1$ and $W_2$ are invariant under $\rho$ and $\mathbb{C}^r = W_1 \oplus W_2$. 
\end{proof}

From part (ii) of Lemma \ref{properties of indecomposable rep}, we conclude: If $\rho:\mathfrak{b}^2\ra \mathfrak{gl}(r,\C)$ is an indecomposable multiplicity-free representation, then there exist $\l_i>0$, $i=1,2$, and $s,t\in\N$ such that $\l_1,\l_1-1,\l_1-2,\hdots,\l_1-s$ and $\l_2,\l_2-1,\l_2-2,\hdots,\l_2-t$ are distinct eigenvalues (with appropriate multiplicities) of $\rho(h_1)$ and $\rho(h_2)$, respectively. Therefore, we set $V_{\theta}$, $\theta=(\theta_1,\theta_2)$, to be the joint eigenspace associated to the joint eigenvalue $(\l_1-\theta_1,\l_2-\theta_2)$ without causing any ambiguity. Take $V_{\theta}$ to be $\{0\}$ if $\l-\theta$ is not a joint eigenvalue of $D_{\rho}$ where $\l=(\l_1,\l_2)$.



For a given multiplicity-free representation $\rho$ of $\mathfrak{b}^2$, we associate a planar graph to it and then relate the indecomposability of $\rho$ to the connectedness of this graph. More precisely, this graph is obtained by taking $\theta=(\theta_1,\theta_2)\in\N\cup\{0\}\times\N\cup\{0\}$ as it's vertices whenever $V_{\theta}$ is a joint eigenspace of $D_{\rho}$ and the edge between two consecutive vertices $\theta$ and $\theta+\varepsilon_j$, $j=1,2$ exists when $\rho(y_j)|_{V_{\theta}}\neq 0$. It turns out that the connectedness of this graph is a necessary condition for $\rho$ to be indecomposable. However, in general, $\rho$ may be decomposable even if the graph associated to it is connected. We list below some properties  of  such graphs corresponding to an indecomposable representation $\rho$. 

\begin{itemize}
\vspace{0.01in}
\item[$P_1:$] For every $1\leq i\leq s$, if for any two $1\leq j_1<j_2\leq t$, both $(\l_1-i,\l_2-j_1)$ and $(\l_1-i,\l_2-j_2)$ are joint eigenvalues of  {\color{blue} $D_{\rho}$}, then for every $j_1\leq j\leq j_2$, so is $(\l_1-i,\l_2-j)$.\vspace{0.05in}
\item[$P_2:$] For every $1\leq j\leq t$, if for any two $1\leq i_1<i_2\leq s$, both $(\l_1-i_1,\l_2-j)$ and $(\l_1-i_2,\l_2-j)$ are joint eigenvalues of {\color{blue} $D_{\rho}$}, then so is $(\l_1-i,\l_2-j)$ for every $i_1\leq i\leq i_2$.\vspace{0.05in}
\item[$P_3:$] For any two consecutive eigenvalues $\lambda$ and $\lambda - 1$ of $\rho(h_1)$, there exists an eigenvalue $\mu$ of $\rho(h_2)$ such that $(\lambda, \mu)$ and $(\lambda - 1, \mu)$ both are joint eigenvalues of $\rho$. \vspace{0.05in}
\item[$P_4:$] If both $\l-\theta$ and $\l-\theta-\varepsilon_j$ are joint eigenvalues of {\color{blue} $D_{\rho}$} then $\rho(y_j)$ takes $V_{\theta}$ onto $V_{\theta+\varepsilon_j}$, $j=1,2$.
\end{itemize} 

Note that the property $P_1$ (resp. $P_2$) indicates that if two elements of $\N\cup\{0\}\times\N\cup\{0\}$ in a vertical (resp. horizontal) line are vertices then every elements of $\N\cup\{0\}\times\N\cup\{0\}$ in between them are also vertices. Property $P_3$ says that if we consider two horizontal strings of vertices, then there exist vertices with common second coordinates in each string and property $P_4$ says that any two consecutive vertices must be joined by an edge. We now show that these  properties actually characterize such representations. 

\begin{thm}\label{Classification of rep}
Let $\rho:\mathfrak{b}^2\ra \mathfrak{gl}(r, \C)$ be a multiplicity-free representation. Then $\rho$ is indecomposable if and only if $\rho$ satisfies $P_1$, $P_2$, $P_3$ and $P_4$.
\end{thm}

\begin{proof}
We first show that the representation $\rho$ satisfies $P_1$, $P_2$, $P_3$ and $P_4$ with the help of mathematical induction on the number of distinct eigenvalues of $\rho(h_1)$ assuming that it is indecomposable.

We begin with the case when $\rho(h_1)$ has only one eigenvalue. It follows that $\rho(y_1)=0$ and consequently, $\rho$ is indecomposable if and only if so is $\rho|_{\{0\}\oplus\mathfrak{b}}$. It is then easy to verify that $\rho$ satisfies $P_1$, $P_2$, $P_3$ and $P_4$. 

Assume that if $\rho(h_1)$ has $k$ distinct eigenvalues then $\rho$ satisfies $P_1$, $P_2$, $P_3$ and $P_4$. Let $\rho(h_1)$ have $k+1$ distinct eigenvalues, say, $\l_1,\l_1-1,\hdots,\l_1-k$ with eigenspaces $E_0,E_1,\hdots, E_k$, respectively, for some positive real number $\l_1$. It can be seen from the properties $P_1$, $P_2$, $P_3$ and $P_4$  that
\begin{itemize}
\item $\C^r=E_0\oplus E_1\oplus\cdots\oplus E_k$;
\item $\rho(y_1)$ maps $E_i$ into $E_{i+1}$ and $\rho(y_1)|_{E_k}=0$;
\item $\rho(y_2)$ keeps each $E_i$ invariant.
\end{itemize}
Consider the representation $\tilde{\rho}:=\rho|_{E_1\oplus E_2\oplus\hdots\oplus E_k}$. Clearly, $\tilde{\rho}$ is multiplicity-free. We show that $\tilde{\rho}$ is an indecomposable representation.

On the contrary, assume that $A$ and $B$ are two subspaces of $E_1\oplus\hdots\oplus E_k$ such that both $A$ and $B$ are invariant under $\tilde{\rho}$ and $\C^r=A\oplus B$. Since $\rho$ is multiplicity-free both $A$ and $B$ are spanned by joint eigenspaces of $\rho$.

Let $\Lambda = \{(0, j) : \rho(y_2) \left(V_{(0, j-1)}\right) = 0\,\,\mbox{or}\,\,(\lambda_1, \lambda_2 - j + 1)\,\,\mbox{is not a joint eigenvalue}\}$. For each $(0, j) \in \Lambda$, define $\Lambda_j = \{(0, j+ p) : \mbox{either}\,\,p = 0\,\,\mbox{or if}\,\,p \geq 1,\,\,\mbox{then}\,\,\rho(y_2) \left(V_{(0, j+ p - 1)}\right) \neq \{0\} \}$.

We claim that if there exists $(0, j + k)$ in $\Lambda_j$ such that $\rho(y_1) \left(V_{(0, j+k)}\right)$ is a non-zero element of $A$, then there does not exist any $(0, j + k')$ in $\Lambda_j$ such that $\rho(y_1) \left(V_{(0, j+k')}\right)$ is a non-zero element of $B$.

Assume that $\rho(y_1) \left(V_{(0, j+k)}\right)$ and $\rho(y_1) \left(V_{(0, j+k')}\right)$ are non-zero elements of $A$ and $B$, respectively, for some $(0, j+k), (0, j+k') \in \Lambda_j$.

If $k' < k$ then there exists $s,\,\,k' \leq s \leq k$ such that $\rho(y_2) \left(V_{(1, j+s)}\right) = 0$. Part (i) of Lemma \ref{properties of indecomposable rep} yields that $\rho(y_1) \left(V_{(0, j+t)}\right) = 0$ for every $t \geq s$ such that $(0, j+t) \in \Lambda_j$. But it contradicts that $\rho(y_1) \left(V_{(0,j+k)}\right)\neq \{0\}$ in $A$. On the other hand, for $k<k'$, a similar argument implies that $\rho(y_1)\left(V_{(0, j+k')}\right) = 0,$ contradicting that $\rho(y_1)\left(V_{(0, j+k')}\right)$ is a non-zero element of $B$. This verifies the claim.  

Let $\Lambda_A = \cup_j \Lambda_j$ where the union is taken over all $j$ such that $(0, j) \in \Lambda$ with the property that there exists $(0, j+k)$ in $\Lambda_j$ such that $\rho(y_1) \left(V_{(0, j+k)}\right)\neq \{0\}$ in $A$. Consequently, the subspaces $A_1=A\oplus\left(\oplus_{(0, k)\in \Lambda_A}V_{(0,k)}\right)$ and $B_1=B\oplus\left(\oplus_{(0, l) \in \Lambda \setminus \Lambda_A}V_{(0,l)}\right)$ 
decompose $\rho$.

Thus it shows that $\tilde{\rho}$ is indecomposable. Moreover, since $\tilde{\rho}=\rho|_{E_1\oplus\hdots\oplus E_k}$ and $\rho$ is multiplicity-free so is $\tilde{\rho}$. Consequently, it follows from the induction hypothesis that $\tilde{\rho}$ satisfies $P_1$, $P_2$, $P_3$ and $P_4$.

We now prove that $\rho$ satisfies $P_1$, $P_2$, $P_3$ and $P_4$ as follows.

\textit{Proof of $P_1$:} Suppose that there exists $j$ with $j_0\leq j\leq p_0$ such that $(\l_1,\l_2-j)$ is not a joint eigenvalue of $D_{\rho}$. Set $l=\text{min}\{j':j_0\leq j'\leq p_0,~\text{and}~\rho(y_2)|_{V_{(0,j')}}=0\}$. Note that the set $\{j':j_0\leq j'\leq p_0,~\text{and}~\rho(y_2)|_{V_{(0,j')}}=0\}$ is non-empty, otherwise, $(\l_1,\l_2-j)$ would be a joint eigenvalue of $D_{\rho}$. Now either $\rho(y_2)|_{V_{(1,l)}}=0$ or $\rho(y_2)|_{V_{(1,l)}}\neq 0$.

If $\rho(y_2)|_{V_{(1,l)}}=0$, we have from the induction hypothesis that none of $(\l_1-1,\l_2-j')$ for $j'\geq l+1$ is a joint eigenvalue of  $D_{\rho}$. Consequently, the subspaces $A_0=\oplus_{j'\geq l}V_{(0,j')}$ and $B_0=(E_1\oplus\cdots\oplus E_k)\oplus\left(\oplus_{j'\leq l}V_{(0,j')}\right)$ become reducing subspaces for $\rho$.

On the other hand, if $\rho(y_2)|_{V_{(1,l)}}\neq 0$ it follows from Lemma \ref{properties of indecomposable rep} that $\rho(y_1)|_{V_{(0,j')}}= 0$ for $j'\leq l$ implying that the subspaces $A_1=\oplus_{j'\leq l}V_{(0,j')}$ and $B_1=(E_1\oplus\cdots\oplus E_k)\oplus\left(\oplus_{j'\geq l}V_{(0,j')}\right)$ decompose $\rho$.

Thus for all $j_0\leq j\leq p_0$, $(\l_1,\l_1-j)$ is a joint eigenvalue for  $D_{\rho}$ completing the proof of $P_1$.

\textit{Proof of $P_2$:} Since $\rho$ satisfies $P_1$ the induction hypothesis together with Lemma \ref{properties of indecomposable rep} yield that $\rho$ satisfies $P_2$.

\textit{Proof of $P_3$:} Suppose that there does not exist any $\mu$ such that $(\lambda_1 , \mu)$ and $(\lambda_1 - 1, \mu)$ are joint eigenvalues of $D_{\rho}$. It clearly follows that both the subspaces $E_1 \oplus E_2 \oplus \cdots \oplus E_{k}$ and $E_0$ are invariant under $\rho$ contradicting that $\rho$ is indecomposable.

\textit{Proof of $P_4$:} A similar reasoning as in the cases of $P_1$ and $P_2$ shows  that $\rho$ is decomposable whenever both $\l-\theta$ and $\l-\theta-\varepsilon_i$ are joint eigenvalues of $D_{\rho}$ such that $\rho(y_i)|_{V_{\theta}}=0$ for $i=1,2$.


For the converse, let us assume that $\rho$ is decomposable and $\rho$ satisfies $P_1$, $P_2$, $P_3$ and $P_4$. Let $A$ and $B$ be two complementary subspaces of $\C^r$ such that both $A$ and $B$ are invariant under $\rho$. Since $\rho$ is multiplicity-free both $A$ and $B$ are spanned by joint eigenspaces of $\rho$. 

Let $l =$ min $\{ i : $ there exists $k$ such that $(\lambda_1 - i, \lambda_2 - k)$ is a joint eigenvalue of $D_{\rho}$ and the corresponding eigenvector is in $B\}$.

Suppose $l > 0$. Since $\rho$ satisfies $P_3$ there exists $p$ such that both $(\lambda_1 - l + 1, \lambda_2 - p)$ and $(\lambda_1 - l, \lambda_2 - p)$ are joint eigenvalues of $D_{\rho}$. Then the condition $P_4$ together with the fact that the joint eigenvector corresponding to the joint eigenvalue $(\lambda_1 - l + 1, \lambda_2 - p)$ is in $A$ implying that the joint eigenvector corresponding to the joint eigenvalue $(\lambda_1 - l, \lambda_2 - p)$ is also in $A$. If $p < k$ the conditions $P_1$ and $P_4$ together imply that the joint eigenvector corresponding to the joint eigenvalue $(\lambda_1 - l, \lambda_2 - k)$ is in $A$ which is a contradiction. If $p > k$ again the conditions $P_1$ and $P_4$ imply that the joint eigenvector corresponding to the joint eigenvalue $(\lambda_1 - l, \lambda_2 - p)$ is in $B$ which is again a contradiction. A similar contradiction is obtained in the case of $l = 0$ as well. 

Thus we have shown that $\rho$ cannot satisfy the conditions $P_1, P_2, P_3$ and $P_4$ simultaneously whenever $\rho$ is decomposable.
\end{proof}

The following corollary describes a multiplicity-free indecomposable representation of $\mathfrak{b}^2$. The proof is an easy consequence of Lemma \ref{properties of indecomposable rep} and the property $P_4$.

\begin{cor}\label{structure of representation}
Suppose $\rho : \mathfrak{b}^2 \to \mathfrak{gl}(r, \C)$ is a multiplicity-free indecomposable representation such that $\lambda, \lambda - 1$ are eigenvalues of $\rho(h_1)$ with multiplicities $m_1, m_2$, respectively. If $(\lambda, \mu - j_1 - l)$, $0 \leq l \leq m_1$ and $(\lambda - 1, \mu - j_2 - k)$, $0 \leq k \leq m_2$ are joint eigenvalues of $D_{\rho}$, then we have $j_2 \leq j_1$ and $j_2 + m_2 \leq j_1 + m_1$.
\end{cor}

The theorem below provides a characterization of all the irreducible {\h}s over $\D^2$ homogeneous \w $\mob^2$ such that the associated representations are multiplicity-free.

\begin{thm}\label{includecurvature}
Let $J:\widetilde{\mob}^2\ra\text{GL}(r,\C)$ be the cocycle corresponding to irreducible {\h} over $\D^2$ which are homogeneous \w $\mob^2$. Assume that the associated representation is multiplicity-free. Then $J$ takes the form
$$J(\tilde{g},\bz)=\exp\left(\sum_{i=1}^2\dfrac{g_i''(z_i)}{2g_i'(z_i)}\rho(y_i)\right)\exp\left(\sum_{i=1}^2(-\log(g_i'(z_i))\rho(h_i))\right)$$
where $\rho : \mathfrak{b}^2 \to \mathfrak{gl}(r, \C)$ is a multiplicity-free indecomposable representation such that $V_{\theta}$ is the joint eigenspace of $D_{\rho}$ $=(\rho(h_1),\rho(h_2))$ associated to the joint eigenvalue $(\l_1-\theta_1,\l_2-\theta_2)$ and $\rho(y_i)V_{\theta}=V_{\theta+\varepsilon_i}$, $i=1,2$ for $\theta\in \{(i,j_i),\hdots,(i,j_i+m_i):0\leq i\leq s,j_i\geq j_{i'},j_i+m_i\geq j_{i'}+m_{i'},~\text{for}~i\leq i'\}$. Here $\l_2$ is the largest eigenvalue of $\rho(h_2)$, and $s$ and $m_i'$s are natural numbers such that $\l_1,\l_1-1,\hdots,\l_1-s$ are eigenvalues of $\rho(h_1)$ with multiplicities $m_0,m_1,\hdots,m_s$, respectively.
\end{thm}

\begin{proof} For the proof, we merely combine 
 $$[\rho(y_1), \rho(h_2)]=[\rho(y_1), \rho(y_2)]=[\rho(h_1), \rho(h_2)]=0,$$
with  Equation \eqref{cocyclecomp3} and Corollary \ref{structure of representation}.
\end{proof}
\begin{rem} Clearly, any one dimensional representation of $\mathfrak {b}^n$  is  obtained by taking tensor product of the one dimensional representations of $\mathfrak  b$. This provides an  independent validation  of   Theorem \ref{rank one charaterization}.
\end{rem}

\section{Irreducible Homogeneous tuples in $\mathrm B_2(\mathbb D^n)$}\label{Section 6}
In this section, we describe all irreducible homogeneous tuples in $\mathrm B_2(\mathbb{D}^n)$ with respect to the group $\mathsf G,$ which is taken to be either $\mob^n$ or  $\Aut(\mathbb{D}^n)$. All irreducible tuples in $\mathrm B_2(\mathbb{D}^n)$ which are homogeneous with respect to $\mob^n$ have been described. However, we show that there is no irreducible $\Aut(\mathbb{D}^n)$ - homogeneous tuple in $\mathrm B_2(\mathbb{D}^n)$. 

\begin{defn}\label{Def of irreducible tuple}
If the operators $T_1, \ldots, T_n$ have no common reducing subspace, that is, there is no projection that commutes with all of them, then we say that the $n$ - tuple $(T_1, T_2,\ldots, T_n)$ is irreducible. \end{defn}

Let $\l>0$ and $\mathbb{A}^{(\l)}$  denote the reproducing kernel Hilbert space consisting of holomorphic functions on the open unit disc $\mathbb D$ determined by the kernel  $K^{(\l)}(z, w) = (1 - z\bar{w})^{-\l}$ defined on $\mathbb D.$  Also, let  $M^{(\l)}$ denote the operator of multiplication by the coordinate function $z$  on $\mathbb{A}^{(\l)}$. Finally, let $\mathbb{A}^{(\lambda, \mu)}$  be the reproducing kernel Hilbert space determined by the kernel 
$$K^{(\lambda, \mu)}(z, w) = \begin{pmatrix}
    \frac{1}{(1 - z\bar{w})^{\lambda}} & \frac{z}{(1 - z\bar{w})^{\lambda+1}} \\
    \frac{\bar{w}}{(1 - z\bar{w})^{\lambda+1}} & \frac{\frac{1}{\lambda} + \mu + z\bar{w}}{(1 - z\bar{w})^{\lambda+2}}
\end{pmatrix}$$
defined on $\mathbb{D}.$  The operator $M^{(\lambda, \mu)}$ is the multiplication by the coordinate function $z$  on $\mathbb{A}^{(\lambda, \mu)}.$ It is well known (cf. Proposition 4, \cite{CADSR}) that any homogeneous operators in $\mathrm B_1 (\mathbb{D})$ are unitarily  equivalent to ${M^{(\l)}}^*$ for some $\l > 0,$  and every irreducible homogeneous operator in $\mathrm B_2 (\mathbb{D})$ must be unitarily equivalent to ${M^{(\lambda, \mu)}}^*$ for some $\lambda, \mu > 0$ \cite[Theorem 4.1]{HVBCDO}, \cite[Corollary 4.1]{MFHOCD}. 

We prove that the $n$ - tuple $(M_{z_1}, M_{z_2}, \ldots , M_{z_n})$ of multiplication operators by the coordinate functions acting on the Hilbert space $\mathbb{A}^{(\l_1,\mu)}\otimes\mathbb{A}^{(\l_2)}\otimes\cdots \otimes \mathbb{A}^{(\l_{n})}\subseteq \mbox{Hol}(\mathbb D^n, \mathbb C^2)$ is irreducible for $\mu>0$ and any tuple of positive real numbers $\bl=(\l_1,\hdots,\l_n)$. First, we prove a useful lemma.

\begin{lem}\label{Projections commuting with an irr tensor with identity op}
Let $H_1$ and $H_2$ be two Hilbert spaces and $T_i$ be an irreducible operator on $H_i$ for $i = 1, 2$. Suppose $P$ is a projection defined on $H_1 \otimes H_2$.

\noindent(a) If $P$ commutes with $I \otimes T_2$, then there exists a projection $P_1$ defined on $H_1$ such that $P = P_1 \otimes I$.

\noindent(b) If $P$ commutes with $T_1 \otimes I$, then there exists a projection $P_2$ defined on $H_2$ such that $P = I \otimes P_2$.
\end{lem}
\begin{proof}
(a) Assume that $\dim H_1 = N$, where $N$ can be $\infty$. Let $\{e_i : 1 \leq i \leq N\}$ be an orthonormal basis of $H_1$. Define $U : H_1 \otimes H_2 \to \bigoplus_{i = 1}^N H_2$ by $U(e_1 \otimes y) = (0,0,\ldots, y, 0, \ldots, 0),\,\,y \in H_2,$ where $y$ is in the $i$ - th position. Then $U$ is a unitary operator and $U(I \otimes T_2)U^* = \bigoplus_{i = 1}^N T_2$.

Let $\tilde{P} = UPU^*$. Suppose $( \!(\tilde{P}_{ij}) \!)_{i,j}^N$ is the matrix representation of $\tilde{P}$ in the Hilbert space $\oplus_{i = 1}^N H_2$ where $\tilde{P}_{ij}$ is an operator on $H_2$. Since $\tilde{P}$ is a projection, it is evident that $\tilde{P}_{ij}^* = \tilde{P}_{ji}$ for all $i, j$.

Since $P$ and $I \otimes T_2$ commute, the operators $\tilde{P}$ and $\oplus_{i = 1}^N T_2$ also commute. This implies that $\tilde{P}_{ij}$ commutes with $T_2$ for each $i, j$. Since $T_2$ is irreducible and both $\tilde{P}_{ij}^*$ ($= \tilde{P}_{ji}$) and $\tilde{P}_{ij}$ commute with $T_2$, it can be seen that $\tilde{P}_{ij} = \alpha_{ij}I$ for some $\alpha_{ij} \in \mathbb{C}$. 

Thus we have $\tilde{P} = ( \!(\alpha_{ij} I) \!)$. Let $P_1$ be the operator on $H_1$ whose matrix representation with respect to the orthonormal basis $\{e_i : 1 \leq i \leq N\}$ is $\left(\! (\alpha_{ij} )\! \right)$. Since $\tilde{P} = \left( \!(\alpha_{ij} I) \!\right)$ is a projection, it follows that $P_1$ is also a projection and $P = P_1 \otimes I$.

(b) Let $V : H_1 \otimes H_2 \to H_2 \otimes H_1$ be the unitary operator defined by $V(h_1 \otimes h_2) = h_2 \otimes h_1,\,h_1 \in H_1,\,h_2 \in H_2$. Conjugating $P$ and $T_1 \otimes I$ by $V$ and applying (a), the proof of (b) follows.
\end{proof}

\begin{prop}\label{irredicibility of the tuple in rank 2}
The $n$ - tuple $(M_{z_1}, M_{z_2}, \ldots , M_{z_n})$ of multiplication operators by the coordinate functions acting on the Hilbert space $\mathbb{A}^{(\l_1,\mu)}\otimes\mathbb{A}^{(\l_2)}\otimes\cdots \otimes \mathbb{A}^{(\l_{n})} \subseteq \mbox{Hol}(\mathbb D^n, \mathbb C^2)$ is irreducible. 
\end{prop}

\begin{proof}
Evidently, the $n$ - tuple $(M_{z_1}, M_{z_2}, \ldots , M_{z_n})$ is simultaneously unitarily equivalent to the tuple 
$$(M^{(\l_1,\mu)} \otimes I \otimes \dots \otimes I \otimes I,I\otimes M^{(\l_2)}\otimes I\otimes \dots\otimes I,\ldots,I \otimes \dots \otimes I \otimes M^{(\l_n)})$$
acting on $\mathbb{A}^{(\l_1,\mu)}\otimes\mathbb{A}^{(\l_2)}\otimes\cdots \otimes \mathbb{A}^{(\l_{n})}$.

Let $P$ be a projection which commutes with $I \otimes \dots \otimes I \otimes M^{(\l_n)}$. Then there exists a projection $P_1$ defined on $\mathbb{A}^{(\l_1,\mu)} \otimes\mathbb{A}^{(\l_2)}\otimes \dots \otimes \mathbb{A}^{(\l_{n-1})}$ such that $P = P_1 \otimes I$ by virtue of  Lemma \ref{Projections commuting with an irr tensor with identity op}. 
Now $P$ commutes with $I \otimes \dots \otimes I\otimes M^{(\l_{n-1})}\otimes I$ implying that $P_1$ commutes with $I \otimes \dots \otimes I\otimes M^{(\l_{n-1})}$. Again applying Lemma \ref{Projections commuting with an irr tensor with identity op}, we obtain a projection $P_2$ such that $P_1 = P_2\otimes  I$.

Continuing in this manner, we see that $P = P_{n-1} \otimes I \otimes \dots \otimes I,$ where $P_{n-1}$ is a projection  defined on $\mathbb{A}^{(\lambda_1, \mu)}$ and it commutes with $M^{(\lambda_1, \mu)}$. Since $M^{(\lambda_1, \mu)}$ is irreducible $P_{n-1}$ is either $0$ or $I$.  This proves that the given tuple is irreducible.
\end{proof}

Recall that $D_{\alpha}^+$ is the holomorphic Discrete series representation of $\mob$ on $\mathbb{A}^{(\alpha)}$, $\a>0$, and $D_{\lambda, \mu}^+$ is the multiplier representation of $\mob$ on $\mathbb{A}^{(\lambda, \mu)}$ given by the cocyle
$$J(\tilde{g}, z) = \begin{pmatrix}
    (g'(z))^{\frac{\lambda}{2}} & 0 \\
    \frac{g''(0)}{2(g'(0))^{\frac{3}{2}}}(g'(z))^{\frac{\lambda + 1}{2}} & (g'(z))^{\frac{\lambda+2}{2}}
\end{pmatrix}.$$
It is easy to see that the $n$ - tuple of multiplication by the coordinate functions $(M_{z_1}, \ldots , M_{z_n})$ acting on the Hilbert space $\mathbb{A}^{(\l_1,\mu)}\otimes\mathbb{A}^{(\l_2)}\otimes\cdots \otimes \mathbb{A}^{(\l_{n})}$ is homogeneous under the action of $\mob^n$ with associated representation $D_{\l_1,\mu}^+ \otimes D_{\l_2}^+ \otimes \cdots \otimes D_{\l_n}^+$ of $\mob^n$. However, it turns out that these $n$ - tuples of multiplication operators are not homogeneous with respect to $\aut(\D^n)$ as shown in the following theorem.

\begin{thm}\label{Product kernels are not homogeneous under the full aut gp}
Let $\l_i, \mu$ be positive real numbers where $i = 1, 2,\ldots, n$. The $n$ - tuple of multiplication by the coordinate functions $(M_{z_1}, M_{z_2}, \ldots , M_{z_n})$ acting on the Hilbert space $\mathbb{A}^{(\l_1,\mu)}\otimes\mathbb{A}^{(\l_2)}\otimes\cdots \otimes \mathbb{A}^{(\l_{n})}$
is not homogeneous under the action of $\Aut(\mathbb{D}^n)$.
\end{thm}

\begin{proof}
The reproducing kernel of the Hilbert space $\mathbb{A}^{(\l_1,\mu)}\otimes\mathbb{A}^{(\l_2)}\otimes\cdots \otimes \mathbb{A}^{(\l_{n})}$ is
$$K^{(\bl, \mu)}(\bz, \bw) = \begin{pmatrix}
    \frac{1}{(1 - z_1 \bar{w}_1)^{\lambda_1}} & \frac{z_1}{(1 - z_1 \bar{w}_1)^{\lambda_1+1}} \\
    \frac{\bar{w}_1}{(1 - z_1 \bar{w}_1)^{\lambda_1+1}} & \frac{\frac{1}{\lambda_1} + \mu + z_1\bar{w}_1}{(1 - z_1 \bar{w}_1)^{\lambda_1+2}}
\end{pmatrix}\displaystyle \prod_{i = 2}^{n} \frac{1}{(1 - z_i \bar{w_i})^{\l_{i}}} $$ where $\bl=(\l_1,\hdots,\l_n)$.
Since the $n$ - tuple of multiplication by coordinate functions is homogeneous, for each $\tilde{g} \in  \widetilde{\Aut(\mathbb{D}^n)}$, there exists a holomorphic map $J_{\tilde{g}} : \mathbb{D}^n \rightarrow \text{GL}(2, \mathbb{C})$ such that $K^{(\bl, \mu)}$ is quasi-invariant with respect to $J$. It follows from Lemma \ref{curvature at 0 is diagonal} that $\mathcal{K}^{(\bl,\mu)}_{11}(0)$ and $\mathcal{K}^{(\bl,\mu)}_{nn}(0)$ are similar. A straightforward computation shows that
$$\mathcal{K}^{(\bl,\mu)}_{11}(0) = \begin{pmatrix}
    \lambda_1 - \left(\frac{1}{\lambda_1}-\mu^2 \right)^{-1} & 0\\
    0 & \lambda_1 + 2 + \left(\frac{1}{\lambda_1}-\mu^2 \right)^{-1}
\end{pmatrix}\,\, \mbox{and}\,\,\mathcal{K}^{(\bl,\mu)}_{nn}(0) = \begin{pmatrix}
    \l_n & 0\\
    0 & \l_n
\end{pmatrix} .$$
This implies that $\mathcal{K}^{(\bl, \mu)}_{11}(0)$ and $\mathcal{K}^{(\bl, \mu)}_{nn}(0)$ can not be similar. Consequently, the $n$ - tuple $(M_{z_1}, M_{z_2}, \ldots , M_{z_n})$
can not be homogeneous under the action of $\Aut(\mathbb{D}^n)$.
\end{proof}

So far, we have obtained a class of irreducible homogeneous tuples in $\mathrm B_2(\mathbb{D}^n)$ under the action of $\mob^n$. As a result,  the {\h}s associated to these $n$ - tuples are  irreducible. Moreover, from Theorem \ref{grpactvec} it is clear that these vector bundles admit an action of the universal covering group $\tilde{\mathsf G}$ of $\mathsf G$ where $\mathsf G$ is the group $\mob^n$. We complete this section by showing that any irreducible homogeneous tuples in $\mathrm B_2(\mathbb{D}^n)$ under the action of $\mob^n$ has to be unitarily equivalent to some element in this class. In this regard, we first prove the following lemma which will be used to classify all cocycles associated to irreducible {\h}s over $\D^n$ of rank $2$.

\begin{lem}\label{Description of all indecomposable rep}
Suppose $\rho : \mathfrak{b}^n \rightarrow \mathfrak{gl}(2,\C)$ is a two dimensional indecomposable representation such that $\rho (h_i)$ is diagonalizable for all $i$. Then $\rho|_{\mathfrak{b}^1}$ is indecomposable where $\mathfrak{b}^1=\mathfrak{b}\oplus\{0\}\oplus\cdots\oplus\{0\}$. Furthermore, $\rho (h_j) = \alpha_j I_2$ and  $\rho (y_j) = 0$ for all $j=2,\hdots,n$ where $\alpha_j \in \mathbb{C}$.
\end{lem}

\begin{proof}
We begin by pointing out that $\rho(h_1)$, $\rho(h_2), \ldots, \rho(h_n)$ are simultaneously diagonalizable by means of the fact that $h_i$ and $h_j$ commute, for $i\neq j$, $i,j=1,\hdots,n$. Since $\rho$ is indecomposable there exists $1\leq i\leq n$ such that $\rho(y_i)\neq 0$. Without loss of generality, let us assume that $i=1$. It follows from the identity $[\rho(h_1),\rho(y_1)]=-\rho(y_1)$ that $\rho(h_1)$ has distinct eigenvalues. Consequently, $\rho(y_j)=0$, $j=2,\hdots,n$ since $[\rho(h_1),\rho(y_j)]=0$ and $[\rho(h_j),\rho(y_j)]=-\rho(y_j)$, for $j=2,\hdots,n$. Thus $\rho|_{\mathfrak{b}^1}$ is indecomposable. Furthermore, it follows from $[\rho(h_j),\rho(y_1)]=0$ that $\rho (h_j) = \alpha_j I_2$ for some $\alpha_j \in \mathbb{C}$ for all $j=2,\hdots,n$.
\end{proof}

The following proposition has the description of all the cocycles of the group $\widetilde{\text{SU}(1,1)^n}$, which correspond to irreducible {\h} over $\mathbb D^n$ of rank $2$.

\begin{prop}\label{Description of all cocyle}
Let $J : \widetilde{\text{SU}(1,1)^n} \times \mathbb{D}^n \rightarrow \text{GL}(2, \mathbb{C})$ be a cocycle such that $J(\tilde{k}, 0)$ is diagonal for all $\tilde{k} \in \tilde{\mathbb{K}}$. Then there exists $\l_i\geq 0$, $i = 1, 2,\ldots, n$ and $\lambda_1 \neq 0$ such that 
$$J(\tilde{g}, \bz) = \begin{pmatrix}
\left(g_1^{'}(z_1)\right)^{\lambda_1} & 0\\
\frac{g_{1}^{''}(0)}{2\left(g_{1}^{'}(0)\right)^{\frac{3}{2}}}\left(g_1^{'}(z_1)\right)^{\lambda_1 + \frac{1}{2}} & \left(g_1^{'}(z_1)\right)^{\frac{\lambda_1 + 2}{2}}
\end{pmatrix}\displaystyle \prod_{i=2}^{n} g_{i}^{'}(z_i)^{\l_i} $$
where $\tilde{g}\in\widetilde{\text{SU}(1,1)^n}$ with $p(\tilde{g})= g = (g_1, g_2,\ldots, g_n) \in \text{SU}(1,1)^n$ and $\bz = (z_1, z_2,\ldots, z_n) \in \mathbb{D}^n$. 
\end{prop}

\begin{proof}
Suppose $\rho$ is a two dimensional indecomposable representation of $\mathfrak{b}^n$. Applying  Lemma \ref{Description of all indecomposable rep}, we assume that there exists $\l_i\geq 0$, $i = 1, 2,\ldots, n$ and $\lambda_1 \neq 0$ such that 
$$\rho(h_1) = \begin{pmatrix}
-\lambda_1 & 0\\
0 & -\lambda_1 - 1 
\end{pmatrix},\,\, \rho(y_1) = \begin{pmatrix}
0 & 0\\
1 & 0
\end{pmatrix},~\text{and}~\rho(h_i) = \l_i I_2,\,\, \rho(y_i) = 0,$$ for $i = 2,\ldots, n$. Let $g = \left(\!\!\left(\begin{smallmatrix}
a_i & b_i\\
c_i & d_i
\end{smallmatrix}\right)\!\!\right)_{i=1}^{n} \in \text{SU}(1,1)^n$.
Now the proof follows by substituting the values of $\rho(h_i)$ and $\rho(y_i)$ in \eqref{cocyclecomp3}.
\end{proof}

We denote the cocyle 
$$J(\tilde{g}, \bz) = \begin{pmatrix}
\left(g_1^{'}(z_1)\right)^{\lambda_1} & 0\\
\frac{g_{1}^{''}(0)}{2\left(g_{1}^{'}(0)\right)^{\frac{3}{2}}}\left(g_1^{'}(z_1)\right)^{\lambda_1 + \frac{1}{2}} & \left(g_1^{'}(z_1)\right)^{\frac{\lambda_1 + 2}{2}}
\end{pmatrix}\displaystyle \prod_{i=2}^{n} g_{i}^{'}(z_i)^{\l_i} $$
by $J_{\bl}$ where $\bl = (\l_1, \l_2,\ldots, \l_{n})$.  We find possible values of $\bl$ for which there exists a diagonal matrix $K(0,0)$ such that 
\begin{enumerate}
\item[(a)] the polarization of $K(z,z)$, defined by the equation \eqref{cocyclevsker} is a quasi-invariant kernel with respect to $J_{\bl}$ and 
\item[(b)] the $n$ - tuple of multiplication operators is in $\mathrm B_2(\mathbb D^n).$ 
\end{enumerate}
Suppose that there exists a positive diagonal matrix $K(0,0)$ such that the polarization of $K(z,z)$, defined by the equation \eqref{cocyclevsker} is a quasi-invariant kernel with respect to $J_{\bl}$ under the action of the group $\mob^n$. The function $(1 - z_i \bar{w}_i)^{-\l_i}$ then defines a positive definite kernel on $\mathbb{D}$, for each $i = 2,\ldots, n$ implying that $\l_i$ must be positive for each $i$. Also, it can be seen that the polarization of 
$$J_{\bl}((\tilde{g_z},\tilde{e},\ldots,\tilde{e});(z,0,\ldots,0)) K(0,0) J_{\bl}((\tilde{g_z},\tilde{e},\ldots,\tilde{e});(z,0,\ldots,0))^{*}$$
is a positive definite kernel on $\mathbb{D}$ where $\tilde{e}$ is the identity element of $\widetilde{\mob}$ and $\tilde{g}_z\in\widetilde{\mob}$ such that $p(\tilde{g}_z)=g_z$ maps $z$ to $0$. It has been shown in \cite[Section 4]{MFHOCD} that the polarization of 
$$J_{\bl}((\tilde{g_z},\tilde{e},\ldots,\tilde{e});(z,0,\ldots,0)) K(0,0) J_{\bl}((\tilde{g_z},\tilde{e},\ldots,\tilde{e});(z,0,\ldots,0))^{*}$$
is a positive definite kernel on $\mathbb{D}$, only when $\lambda_1$ is positive and $K(0,0)$ is of the form 
$$K(0,0) = \begin{pmatrix}
1 & 0\\
0 & \frac{1}{\lambda_1} + \mu
\end{pmatrix}$$
for some positive real number $\mu$. Thus we have that 
\beq\label{rank two charaterization}K^{(\bl, \mu)}(\bz, \bw)~ =~  \begin{pmatrix}
    \frac{1}{(1 - z_1 \bar{w_1})^{\lambda_1}} & \frac{z_1}{(1 - z_1 \bar{w_1})^{\lambda_1+1}} \\
    \frac{\bar{w_1}}{(1 - z_1 \bar{w_1})^{\lambda_1+1}} & \frac{\frac{1}{\lambda_1} + \mu + z_1\bar{w_1}}{(1 - z_1 \bar{w_1})^{\lambda_1+2}}
\end{pmatrix} \displaystyle \prod_{i = 2}^{n} \frac{1}{(1 - z_i \bar{w_i})^{\l_{i}}} \eeq
is the only kernel on $\mathbb{D}^n$ such that the $n$ - tuple of multiplication operators is in $\mathrm B_2(\mathbb{D}^n)$ and homogeneous under the action of $\mob^n$ where $\bl = (\l_1, \l_2,\ldots, \l_{n})$ is a tuple of positive real number and $\mu > 0$. Thus we have proved the following theorem.
\begin{thm}\label{homogeneous rank 2 bundle}
Each $\mbox{M\"{o}b}^n$ homogeneous tuple of operators in $\mathrm B_2(\D^n)$ is unitarily equivalent with the adjoint of the tuple of multiplication operators on the reproducing kernel Hilbert space $\mathbb{A}^{(\l_1,\mu)}\otimes\mathbb{A}^{(\l_2)}\otimes\cdots \otimes \mathbb{A}^{(\l_{n})}$ determined by the kernel $K^{(\bl, \mu)}$, where $\bl = (\lambda_1, \lambda_2, \ldots, \lambda_n)$ is a tuple of positive real numbers and $\mu > 0$.  
\end{thm}
As a consequence of Theorem \ref{Product kernels are not homogeneous under the full aut gp}, it follows that there is no rank two quasi-invariant kernel on $\mathbb{D}^n$ for the action of $\Aut(\mathbb{D}^n)$.

\section{Irreducible Homogeneous tuples in $\mathrm B_3(\mathbb D^n)$}\label{Section 7}
In this section, we study homogeneous operator tuples in $\mathrm B_3(\D^n)$. Note that a natural class of homogeneous $n$ - tuples of operators in $\mathrm B_3(\D^n)$ can be obtained by taking tensor product of homogeneous operators in $\mathrm B_3(\D)$ with  those in $\mathrm B_1(\D^{n-1})$. Thus this class of irreducible homogeneous $n$ - tuples of operators in $\mathrm B_3(\D^n)$ consists of the adjoint of the multiplication operators by coordinate functions on the \rkhs $\H_{\hat{K}}$ with the reproducing kernel 
\beq\label{eqn14}
\hat{K}(\bz,\bw):=\mathbf{B}^{(\l_1,\bmu)}(z_1,w_1)\cdot\prod_{i=2}^n(1-z_i\ov{w}_i)^{-\l_i},\eeq
where $\mathbf{B}^{(\l_1,\bmu)}(z_1,w_1)$ is the reproducing kernel on $\D$ obtained in \cite[Section 3]{HOHS}, $\l_1 >0$ and $\bmu=(1,\mu_1,\mu_2)$ with $\mu_1,\mu_2>0$. These are demonstrably $\widetilde{\mbox{M\"{o}b}^n}$ - homogeneous. However, this section is devoted in finding a large class of $n$ - tuples of operators  in $\mathrm B_3(\D^n)$, which are not of this form. In the following subsections, two continuous family of Hilbert spaces, namely \textbf{Type I} and \textbf{Type II} Hilbert spaces, are constructed explicitly by prescribing two types of $\G$ maps. Among them, although $\G^{(1)}$ map $-$ associated to \textbf{Type I} Hilbert spaces $-$ is the several variables analogue of that in \cite{HOHS}, $\G^{(2)}$ map $-$ associated to \textbf{Type II} Hilbert spaces $-$ is new.

\subsection{Construction of \textbf{Type I} Hilbert spaces} For $\l>0$, let $\A^{(\l)}$ be the Hilbert space of holomorphic functions on the open unit disc $\D$ with the reproducing kernel $(1-z\ov{w})^{-\l}$. 

\textbf{(I)}  Let $\l_1,\hdots,\l_n>0$,  $i=1,2,3$ and $\bl=(\l_1,\hdots,\l_n)$. Consider the Hilbert spaces $\A_i^{(\bl)}$ of holomorphic functions on $\D^n$ which are, by definition,
$$\A_i^{(\bl)}:=\otimes_{j=1}^n\A^{(\l_j+2\delta_{i-1,j})}$$ where $\delta_{i-1,j}$ is $1$ if $i-1=j$ and $0$ otherwise. Let $\Gamma^{(1)}_i$, $i=1,2,3$, be the linear map on the Hilbert space $\A_i^{(\bl)}$ taking values in Hol$(\D^n,\C^3)$ defined as follows:
$$\G^{(1)}_1(\bf_1)=\left(
    \begin{array}{c}
       \bf_1 \\
       \l_1^{-1}\del_1\bf_1 \\
     \l_2^{-1}\del_2\bf_1 \\
    \end{array}
    \right),~~~
\G^{(1)}_2(\bf_2)=\left(
    \begin{array}{c}
       0 \\
       \bf_2 \\
       0  \\
    \end{array}
    \right),~~~\text{and}~~~
\G^{(1)}_3(\bf_3)=\left(
    \begin{array}{c}
       0 \\
       0 \\
       \bf_3  \\
    \end{array}
    \right)    
$$
where $\bf_i\in\A_i^{(\bl)}$. Note that each $\G^{(1)}_i$, $i=1,2,3$, is a one to one linear map and it induces an inner product on $\G^{(1)}_i(\A_i^{(\bl)})$, namely,
$$\<\G^{(1)}_i(\bf_i),\G^{(1)}_i(\bg_i)\>:=\<\bf_i,\bg_i\>_{\A_i^{(\bl)}}, \,\,i=1,2,3,$$ for all $\bf_i,\bg_i\in\A^{(\bl)}_i$. Then by definition, $\G_i$, $i=1,2,3$, is an isometry. 

Thus it gives rise to the Hilbert spaces $\G^{(1)}_i(\A_i^{(\bl)})$ to be denoted by $\bA_i^{(\bl)}\subset \text{Hol}(\D^n,\C^3)$. Furthermore, the point evaluation, $\G^{(1)}_i(\bf_i)\mapsto\G^{(1)}_i(\bf_i)(\bz)$, for $\bz\in\D^n$ and $i=1,2,3$, are continuous which makes $\bA_i^{(\bl)}$ a \rkhs with the reproducing kernel, say $K_i^{(\bl)}$.

Let us now consider the Hilbert space $\A^{(\bl)}(\D^n):=\oplus_{i=1}^3\A_i^{(\bl)}$ and define the linear map $\G^{(1)}:\A^{(\bl)}(\D^n)\longrightarrow \text{Hol}(\D^n,\C^3)$ as follows:
$$\G^{(1)}(\bf_1,\bf_2,\bf_3):= \sum_{i=1}^3 \G^{(1)}_i(\bf_i).$$ We note that $\G^{(1)}$ is also an one-to-one linear map onto it's image. For any $\mu_1,\mu_2>0$, define an inner product on $\G^{(1)}(\A^{(\bl)}(\D^n))$ making it a Hilbert space, $\bA^{(\bl,\bmu)}(\D^n)$, in the following manner:
$$\<\G^{(1)}(\bf),\G^{(1)}(\bg)\>:=\<\G^{(1)}_1(\bf_1),\G^{(1)}_1(\bg_1)\>+\mu_1^2\<\G^{(1)}_2(\bf_2),\G^{(1)}_2(\bg_2)\>+\mu_2^2\<\G^{(1)}_3(\bf_3),\G^{(1)}_3(\bg_3)\>,$$
$\bf = (\bf_1, \bf_2, \bf_3)$, $\bg = (\bg_1, \bg_2, \bg_3)$ $\in \A^{(\bl)}(\D^n).$ Moreover, $\bA^{(\bl,\bmu)}(\D^n)$ is a \rkhs with the reproducing kernel
\beq\label{eqn1} 
K^{(\bl,\bmu)}(\bz,\bw)=K_1^{(\bl)}(\bz,\bw)+\mu_1^2K_2^{(\bl)}(\bz,\bw)+\mu_2^2K_3^{(\bl)}(\bz,\bw), \text{ for }\bz,\bw\in\D^n.\eeq
In the following proposition, we compute the reproducing kernel $K^{(\bl,\bmu)}$.

\begin{prop}\label{prop2}
The reproducing kernel of the Hilbert space $\bA^{(\bl,\bmu)}(\D^n)$ is 
\begin{flalign*}
K^{(\bl,\bmu)}(\bz,\bw)&=
   D(\bz,\bw)\!\begin{pmatrix}
      1 & z_1 & z_2\\
\ov{w}_1 & \frac{1}{\l_1}+\mu_1^2+z_1\ov{w}_1 & \ov{w}_1z_2\\
\ov{w}_2 & z_1\ov{w}_2 & \frac{1}{\l_2}+\mu_2^2+z_2\ov{w}_2\\
    \end{pmatrix}\\
    &\times D(\bz,\bw)\prod_{j=1}^n(1-z_j\ov{w}_j)^{-\l_j-2\delta_{1j}-2\delta_{2j}}
\end{flalign*}
where $D(\bz,\bw)$ is the $3\times 3$ diagonal matrix diag$((1-z_1\ov{w}_1)(1-z_2\ov{w}_2),(1-z_2\ov{w}_2),(1-z_1\ov{w}_1))$ and $\delta$ is the Kronecker delta function.
\end{prop}

\begin{proof}
Let $K^{(\bl)}(\bz,\bw)=\prod_{j=1}^n(1-z_j\ov{w}_j)^{-\l_j}$ be the reproducing kernel for the Hilbert space $\otimes_{j=1}^n\A^{(\l_j)}$. First, we compute the reproducing kernel $K_1^{(\bl)}$. We claim that 
$$K_1^{(\bl)}(\bz,\bw)=\left(\!\!\left(\l_i^{-1}\l_j^{-1}\del_i\dbar_jK^{(\bl)}(\bz,\bw)\right)\!\!\right)_{i,j=0}^2$$ where $\l_0=1$ and $\del_0f=\dbar_0f=f$ for any function $f$.

Let $\bf \in \otimes_{j=1}^n\A^{(\l_j)}$. Note, for $j=0,1,2$ and $\bw\in\D^n$, that 
$$\<\G_1(\bf),K_1^{(\bl)}(\cdot, \bw)\epsilon_j\>~=~\<\G_1(\bf),\G_1(\l_j^{-1}\dbar_j K^{(\bl)}(\cdot,\bw))\>~=~\<\G_1(\bf)(\bw),\epsilon_j\>$$implying that $K_1^{(\bl)}$ satisfies the reproducing property.

We now observe from the definition of $\G_2$ and $\G_3$ that 
$$K_2^{(\bl)}(\bz,\bw) = \left(\begin{smallmatrix}
0 & 0 & 0\\
0 & \prod_{j=1}^n(1-z_j\ov{w}_j)^{-\l_j-2\delta_{1j}} & 0\\
0 & 0 & 0
\end{smallmatrix}\right)\text{ and }
K_3^{(\bl)}(\bz,\bw) = \left(\begin{smallmatrix}
0 & 0 & 0\\
0 & 0 & 0\\
0 & 0 & \prod_{j=1}^n(1-z_j\ov{w}_j)^{-\l_j-2\delta_{2j}}
\end{smallmatrix}\right).$$
Then substituting $K_1^{(\bl)}$, $K_2^{(\bl)}$ and $K_3^{(\bl)}$ in the equation $\eqref{eqn1}$ the desired form of the reproducing kernel $K^{(\bl,\bmu)}(\bz,\bw)$, for $\bz,\bw \in\D^n$, is obtained.
\end{proof}

We now prove the boundedness of the multiplication operators on $\bA^{(\bl,\bmu)}(\D^n)$ using the following well-known lemma.

\begin{lem}\label{lem3}
For $1\leq j\leq n$, the multiplication operator $M_{z_j}$ is bounded if and only if there exists a positive constant $c_j$ such that \beq\label{inq4}(c_j^2-z_j\ov{w}_j)K(\bz,\bw)\geq 0.\eeq
\end{lem}

For any $\epsilon>0$, the multiplication operator $M^{(\epsilon)}$ on $\A^{(\epsilon)}$, $(M^{(\epsilon)}f)(z)=zf(z)$, is bounded. Consequently, the reproducing kernel $K^{(\epsilon)}$ of the Hilbert space $\A^{(\epsilon)}$ satisfies the inequality \eqref{inq4} in  Lemma \ref{lem3} for any $\epsilon>0$. We use this technique to show that the $n$ - tuple of the multiplication operators by coordinate functions on $\bA^{(\bl,\bmu)}(\D^n)$ are bounded.

\begin{thm}\label{thm5}
The multiplication operators $M_{z_1},\hdots,M_{z_n}$ on $\bA^{(\bl,\bmu)}(\D^n)$ corresponding to the coordinate functions are bounded.
\end{thm}

\begin{proof}
We begin with the observation that the multiplication operators $M_{z_j}$, $3\leq j\leq n$, are bounded since the reproducing kernel $K^{(\bl,\bmu)}(\bz,\bw)$ has the form (Proposition \ref{prop2})
$$
K^{(\bl,\bmu)}(\bz,\bw)=K^{((\l_1,\l_2),(\mu_1,\mu_2))}((z_1,z_2),(w_1,w_2))\prod_{j=3}^n(1-z_j\ov{w}_j)^{-\l_j}$$ where $K^{((\l_1,\l_2),(\mu_1,\mu_2))}((z_1,z_2),(w_1,w_2))$ is the reproducing kernel
\begin{flalign*}
&D(\bz,\bw)\begin{pmatrix}
      1 & z_1 & z_2\\
\ov{w}_1 & \frac{1}{\l_1}+\mu_1^2+z_1\ov{w}_1 & \ov{w}_1z_2\\
\ov{w}_2 & z_1\ov{w}_2 & \frac{1}{\l_2}+\mu_2^2+z_2\ov{w}_2\\
    \end{pmatrix}D(\bz,\bw)(1-z_1\ov{w}_1)^{-\l_1-2}(1-z_2\ov{w}_2)^{-\l_2-2}
\end{flalign*}
where $D(\bz,\bw)$ is the $3\times 3$ diagonal matrix diag $((1-z_1\ov{w}_1)(1-z_2\ov{w}_2),(1-z_2\ov{w}_2),(1-z_1\ov{w}_1))$
and $M_{z_j}$ is bounded on $\A^{(\l_j)}$, for $3\leq j\leq n$. Therefore, it is enough to show that $M_{z_1}$ and $M_{z_2}$ are bounded.

Let $\epsilon>0$ be such that $\epsilon<\frac{\mu_1^2\l_1^2}{1+\l_1\mu_1^2}.$ Note that $$\tilde{\mu}_1^2:=\mu_1^2+\dfrac{1}{\l_1}-\dfrac{1}{\l_1-\epsilon}>0.$$ Consequently, $K^{((\l_1-\epsilon,\l_2),(\tilde{\mu}_1,\mu_2))}((z_1,z_2),(w_1,w_2))$ is a positive definite kernel. Also, it is observed that 
$$K^{((\l_1,\l_2),(\mu_1,\mu_2))}((z_1,z_2),(w_1,w_2))=(1-z_1\ov{w}_1)^{-\epsilon}K^{((\l_1-\epsilon,\l_2),(\tilde{\mu}_1,\mu_2))}((z_1,z_2),(w_1,w_2)).$$
Since the multiplication operator on the Hilbert space whose reproducing kernel is $(1-z_1\ov{w}_1)^{-\epsilon}$ is bounded, it follows that there exists $c_1>0$ such that 
$$(c_1-z_1\ov{w}_1)(1-z_1\ov{w}_1)^{-\epsilon}\geq 0.$$ So multiplying $(c_1-z_1\ov{w}_1)$ with the both sides of the equation above we see that 
$$(c_1-z_1\ov{w}_1)K^{((\l_1,\l_2),(\mu_1,\mu_2))}((z_1,z_2),(w_1,w_2))\geq 0.$$ It follows from Lemma \ref{lem3} that $M_{z_1}$ is bounded and a similar argument with $0<\epsilon<\frac{\mu_2^2\l_2^2}{1+\l_2\mu_2^2}$ yields that $M_{z_2}$ is also bounded.\end{proof}

\subsection{Homogeneity and irreducibility}
In this subsection, we first show that the tuple of multiplication operators $(M_{z_1},\hdots,M_{z_n})$ on $\bA^{(\bl,\bmu)}(\D^n)$ are homogeneous under the action of the group $\mob^n$. The tuple is then shown to be irreducible. The proof follows from establishing that the only self-adjoint projections which commute with the normalized kernel are scalar times the identity operator.

For $g\in\widetilde{\mob}$, $g'(z)^{\a}$ is a real analytic function on the simply connected set $\widetilde{\mob}\times \D$, holomorphic in $z$. Since $g$ is one to one and holomorphic, we also have that $g'(z)^{\a}\neq 0$. Now given $\a\in\R_{>0}$, taking the principal branch of power function when $g$ is near the identity, we can uniquely define $g'(z)^{\a}$ as a real analytic function on $\widetilde{\mob}\times \D$ which is holomorphic on $\D$ for all fixed $g\in\widetilde{\mob}$. The multiplier $j_{\a}(g,z)=g'(z)^{\a}$ defines on $\A^{(\a)}(\D)$ the unitary representation $D^+_{\a}$ as follows:
$$D^+_{\a}(g^{-1})(f):=(g')^{\frac{\a}{2}}(f\circ g), ~~~f\in\A^{(\a)}(\D),~g\in\widetilde{\mob}.$$ As a consequence, for any $\a,\b\in\R_{>0}$, the multiplier $j_{(\a,\b)}((g_1,g_2),(z_1,z_2)):=g'_1(z_1)^{\frac{\a}{2}}g'_2(z_2)^{\frac{\b}{2}}$ defines on $\A^{(\a)}\otimes\A^{(\b)}$ the unitary representation $D^+_{\a}\otimes D^+_{\b}$ of the group $\widetilde{\mob}\times\widetilde{\mob}$ in the following manner:
\beq\label{eqn6}
(D^+_{\a}\otimes D^+_{\b})(g_1^{-1},g_2^{-1})(f_1\otimes f_2) = g'_1(z_1)^{\frac{\a}{2}}(f_1\circ g_1)g'_2(z_2)^{\frac{\b}{2}}(f_2\circ g_2),
\eeq
for $f_1\otimes f_2\in\A^{(\a)}\otimes\A^{(\b)}$, $(g_1,g_2)\in\widetilde{\mob}\times \widetilde{\mob}$ and $(z_1,z_2)\in\D^2$. Thus for the group $\tilde{\mathsf G}=\widetilde{\mob}^n$, we have the unitary representations $\otimes_{j=1}^nD^+_{\l_j+2\delta_{i-1,j}}$ on the Hilbert spaces $\A_i^{(\bl)}$, for $i=1,2,3$, respectively. Therefore, the direct sum of these representations can be transferred to $\bA^{(\bl,\bmu)}(\D^n)$ by the map $\G$. We show that this is a multiplier representation. In this regard, we introduce the notaion 
\beq\label{not7}
D^+_{\bl}:=\oplus_{i=1}^3\left(\otimes_{j=1}^nD^+_{\l_j+2\delta_{i-1,j}}\right).
\eeq
We need a relation between $g''(z)$ and $g'(z)$, for $g\in\widetilde{\mob}$, in the following calculations. The elements of SU$(1,1)$ are the matrices $\left(\!
    \begin{smallmatrix}
      a & b \\
      \ov{b}    & \ov{a} \\
   \end{smallmatrix}
    \!\right)$ with $|a|^2-|b|^2=1$, acting on $\D$ by the fractional linear transformations. The inequalities 
\beq\label{inq8}
|a-1|<\dfrac{1}{2},~~~|b|<\dfrac{1}{2}
\eeq
determine a simply connected neighbourhood $U_0$ of the identiy $e$ in SU$(1,1)$. Under the natural projections, it is diffeomorphic to a neighbourhood $\tilde{U}$ of $e$ in $\widetilde{\mob}$. So we may use $a,b$ satisfying \eqref{inq8} to parametrize $\tilde{U}$. For $g\in\tilde{U}$, $z\in \D$, we have that $g'(z)=(\ov{b}z+\ov{a})^{-2}$ and $g''(z)=-2\ov{b}(\ov{b}z+\ov{a})^{-3}$ which gives the relation 
\beq\label{eqn9}
g''(z)=-2c_gg'(z)^{\frac{3}{2}}
\eeq 
where $c_g$ depends on $g$ real analytically and is independent of $z$ and the meaning of $g'(z)^{\frac{3}{2}}$ is as defined above. Since both sides of the equation \eqref{eqn9} are real analytic, the identity therein remains true on all of $\widetilde{\mob}\times \D$.

\begin{prop}\label{prop10}
The image of $D^+_{\bl}$ under the map $\G$ is a multiplier representation with the multiplier given by 
\[ J(g,\bz):= \begin{pmatrix}
      (g'_1)^{\frac{\l_1}{2}}(g'_2)^{\frac{\l_2}{2}} & 0 & 0\\
 -c_{g_1}(g'_1)^{\frac{\l_1+1}{2}}(g'_2)^{\frac{\l_2}{2}}      & (g'_1)^{\frac{\l_1+2}{2}}(g'_2)^{\frac{\l_2}{2}} & 0 \\
-c_{g_2}(g'_1)^{\frac{\l_1}{2}}(g'_2)^{\frac{\l_2+1}{2}} & 0 & (g'_1)^{\frac{\l_1}{2}}(g'_2)^{\frac{\l_2+2}{2}} \\
    \end{pmatrix}\prod_{j=3}^n (g'_{j})^{\frac{\l_j}{2}}(\bz)
\].
\end{prop}

\begin{proof}
We begin by pointing out that it is enough to show that 
$$\G^{(1)}_i(D^+_{\bl}(g^{-1})\bf)(\bz)=J(g,\bz)\G^{(1)}_i(\bf\circ g)(\bz),~~~\bf\in\A^{(\bl)}_i,\bz\in\D^n,g\in\tilde{\mathsf G}.$$ Let $i=1$, $\bf=f_1\otimes\cdots\otimes f_n\in\A^{(\bl)}_i$ and $g=(g_1,\hdots,g_n)\in\tilde{\mathsf G}$. Then 
\begin{flalign*}
\G^{(1)}_1(D^+_{\bl}(g^{-1})\bf)(\bz) &= \G^{(1)}_1\left(\prod_{j=1}^nD^+_{\l_j}(g_j^{-1})f_j\right)(\bz)\\
&= \G^{(1)}_1\left(\prod_{j=1}^n(g'_j)^{\frac{\l_j}{2}}(f_j\circ g_j)\right)(\bz)\\
&= \begin{pmatrix}
        \prod_{j=1}^n(g'_j)^{\frac{\l_j}{2}}(f_j\circ  g_j)\\
       \l_1^{-1}\del_1\left[(g'_1)^{\frac{\l_1}{2}}(f_1\circ g_1)\right]\prod_{j=2}^n(g'_j)^{\frac{\l_j}{2}}(f_j\circ g_j) \\
       \l_2^{-1}\del_2\left[(g'_2)^{\frac{\l_2}{2}}(f_2\circ g_2)\right]\prod_{j\neq 2}(g'_j)^{\frac{\l_j}{2}}(f_j\circ g_j) \\
    \end{pmatrix}(\bz)\\
&= J(g,\bz)\G^{(1)}_1\left(\prod_{j=1}^n f_j\circ g_j\right)(\bz)\\
&= J(g,\bz)\G^{(1)}_1(\bf\circ g)(\bz)
\end{flalign*}
which shows the equality for $i=1$. Similar calculation, for $i=2$ and $i=3$, leads to the following identities
$$\G^{(1)}_i(D^+_{\bl}(g^{-1})\bf)(\bz)=J(g,\bz)\G^{(1)}_i(\bf\circ g)(\bz),~~~i=2,3.$$
Thus the proof is complete.
\end{proof}

The verification of the assertion in the theorem below follows from Theorem \ref{thm5} and Proposition \ref{prop10}.

\begin{thm}\label{thm11}
The tuple of multiplication operators $(M_{z_1},\hdots,M_{z_n})$ on $\bA^{(\bl,\bmu)}(\D^n)$ corresponding to the coordinate functions are homogeneous under the action of the group $\mob^n$.
\end{thm}

We conclude this subsection by proving that these $n$ - tuples of multiplication operators are irreducible. The lemma below, modelled after Lemma 5.1 in \cite{HOHS}, can be proved exactly in the same way as in the original proof, so it is omitted.

\begin{lem}
Suppose that the tuple of multiplication operators $\M=(M_{z_1},\hdots, M_{z_n})$ on a reproducing kernel Hilbert space $\H$ with the reproducing kernel $K$ is in $\mathrm B_r(\D^n)$. If there exists an orthogonal projection $X$ commuting with the operator tuple $\M$ then 
$$\Phi_X(\bz)K(\bz,\bw)~=~K(\bz,\bw)\ov{\Phi_X(\bw)}^{\text{tr}}$$ for some holomorphic function $\Phi_X:\D^n\ra\C^{r\times r}$ with $\Phi_X^2=\Phi_X$.
\end{lem}

Thus $\M$ on $\bA^{(\bl,\bmu)}(\D^n)$ is irreducible if and only if there is no non-trivial projection $X_0$ on $\C^r$ satisfying 
$$
X_0K^{(\bl,\bmu)}_0(\bz,0)^{-1}K^{(\bl,\bmu)}_0(\bz,\bw)K^{(\bl,\bmu)}_0(0,\bw)^{-1}~=~K^{(\bl,\bmu)}_0(\bz,0)^{-1}K^{(\bl,\bmu)}_0(\bz,\bw)K^{(\bl,\bmu)}_0(0,\bw)^{-1}X_0.  
$$
where $K^{(\bl,\bmu)}_0(\bz,\bw)=K^{(\bl,\bmu)}(0,0)^{-\frac{1}{2}}K^{(\bl,\bmu)}(\bz,\bw)K^{(\bl,\bmu)}(0,0)^{-\frac{1}{2}}$. Let $\hat{K}_0^{(\bl,\bmu)}(\bz,\bw)$, called the normalized kernel, denote the kernel $K^{(\bl,\bmu)}_0(\bz,0)^{-1}K^{(\bl,\bmu)}_0(\bz,\bw)K^{(\bl,\bmu)}_0(0,\bw)^{-1}$.

\begin{thm}\label{thm12}
The tuple of multiplication operators $\M=(M_{z_1},\hdots,M_{z_n})$ on $\bA^{(\bl,\bmu)}(\D^n)$ corresponding to the coordinate functions are irreducible.
\end{thm}

\begin{proof}
Since the normalized kernel $\hat{K}_0^{(\bl,\bmu)}(\bz,\bw)$ is equivalent to $K^{(\bl,\bmu)}(\bz,\bw)$, it is enough to show that the tuple of multiplication operators $(M_{z_1},\hdots,M_{z_n})$ on the \rkhs corresponding to $\hat{K}_0^{(\bl,\bmu)}(\bz,\bw)$ is irreducible. It  amounts to show that if a self-adjoint projection matrix, say $A=(\!(a_{ij})\!)_{i,j=1}^3$ commutes with $\hat{K}_0^{(\bl,\bmu)}(\bz,\bw)$ for all $\bz,\bw\in\D^n$ then $A$ is $\pm I_3$ where $I_3$ is the identity matrix of order $3$.

For $\bz=(z,0,\hdots,0)$ and $\bw=(0,z,0,\hdots,0)$ with $z\in\D$, we have that
\[\hat{K}_0^{(\bl,\bmu)}(\bz,\bw)=
\begin{pmatrix}
      1 & 0 & 0 \\
      0 & 1    &   0 \\
      0 & \sqrt{\dfrac{\a_1}{\a_2}}|z|^2 & 1 \\
    \end{pmatrix}
\]
where $\a_j=\frac{1}{\l_j}+\mu_j^2$, $j=1,2$. Assume that $A$ commutes with $\hat{K}_0^{(\bl,\bmu)}(\bz,\bw)$. Then equating the $(1,3)$, $(1,2)$, $(2,2)$ and $(2,3)$ entries of the commutator of $A$ and $\hat{K}_0^{(\bl,\bmu)}(\bz,\bw)$ to $0$ we have that $a_{12}=a_{13}=a_{23}=0$ and $a_{22}=a_{33}$, respectively. Thus the matrix $A$ becomes a diagonal matrix with $a_{22}=a_{33}$. Now we take $\bz=(z,0,\hdots,0)=\bw$ and compute the $(1,2)$ entry of $\hat{K}_0^{(\bl,\bmu)}(\bz,\bw)$ which is 
$$
(1-|z|^2)^{-\l_1-2}\sqrt{\a_2}\left[\dfrac{z}{\a_1}(1-|z|^2)-\dfrac{z}{\a_1^2}(\a_1+|z|^2)\right]=
-(1-|z|^2)^{-\l_1-2}\dfrac{\sqrt{\a_2}}{\a_1}z|z|^2\left(1+\dfrac{1}{\a_1}\right)
\neq 0.
$$
Now since $A$ commutes with $\hat{K}_0^{(\bl,\bmu)}(\bz,\bw)$ we have that $a_{11}=a_{22}$ implying that $A=cI_{3\times 3}$ for some constant $c$. But since $A$ is a projection $c=\pm1$.
\end{proof}

\subsection{Inequivalence}

We have pointed out in the beginning of this section, the adjoint of the multiplication operators on the \rkhs $\H_{\hat{K}}$ are homogeneous. In this subsection, we show that the multiplication operators by coordinate functions on the \rkhs $\H_{\hat{K}}$ and those on the Hilbert space $\bA^{(\bl,\bmu)}(\D^n)$ are unitarily \textit{inequivalent}. Moreover, the multiplication operators by the coordinate functions on the two Hilbert spaces $\bA^{(\bl,\bmu)}(\D^n)$ and $\bA^{(\bl',\bmu')}(\D^n)$ 
are inequivalent whenever $(\bl,\bmu)\neq(\bl',\bmu')$.

\begin{thm}\label{thm13}
The $n$ - tuple of multiplication operators by coordinate functions on $\H_{\hat{K}}$ and those on $\bA^{(\bl',\bmu')}(\D^n)$ are unitarily inequivalent irrespective of the choice of  $\bl, \bl', \bmu, \bmu'$, where
$$\hat{K}(\bz,\bw):=\mathbf{B}^{(\l_1,\bmu)}(z_1,w_1)\cdot\prod_{i=2}^n(1-z_i\ov{w}_i)^{-\l_i}$$as before.
\end{thm} 

\begin{proof}
Let $\M_{\hat{K}}$ and $\M^{(\bl', \bmu')}$ denote the $n$ - tuple of multiplication operators by the coordinate functions on $\mathcal{H}_{\hat{K}}$ and $\bA^{(\bl', \bmu')}(\D^n)$, respectively. Note that adjoint of both $\M_{\hat{K}}$ and $\M^{(\bl', \bmu')}$ are in $\mathrm B_3(\mathbb{D}^n)$. Therefore, if $\M_{\hat{K}}$ and $\M^{(\bl', \bmu')}$ are unitarily equivalent, then the hermitian holomorphic vector bundles induced by $\M_{\hat{K}}$ and $\M^{(\bl', \bmu')}$ are equivalent. This, in particular, implies from Lemma \ref{curvature at 0 is diagonal} that, for all $1 \leq i, j \leq n$ and for all $z \in \mathbb{D}^n$, $\mathcal{K}^{ij}(z) = \partial_i\left({\hat{K}}(z,z)^{-1}\bar{\partial}_j{\hat{K}}(z,z)\right)$ and $\mathcal{K}^{ij}_{(\bl', \bmu')}(z) = \partial_i\left(K^{(\bl', \bmu')}(z,z)^{-1}{\bar{\partial}}_jK^{(\bl', \bmu')}(z,z)\right)$ are similar.

A routine computation shows that $\mathcal{K}^{22}(0)=\lambda_2 I_3$ and $\mathcal{K}^{22}_{(\bl',\bmu')}(0)$ is a diagonal matrix with diagonal entries $\lambda'_2-((\l'_2)^{-1}+\mu'_2)^{-1}$, $\lambda'_2$, and $\lambda'_2 + ((\l'_2)^{-1}+\mu'_2)^{-1}+2$. Since $\mathcal{K}_{(\bl', \bmu')}^{22}(0)$ has distinct eigenvalues, $\mathcal{K}^{22}(0)$ and $\mathcal{K}^{22}_{(\bl', \bmu')}(0)$ are not similar. Thus $\M_{\hat{K}}$ and $\M^{(\bl', \bmu')}$ are not unitarily equivalent.
\end{proof}

\begin{thm}\label{thm14}
The $n$ - tuple of multiplication operators $\M^{(\bl, \bmu)}$ and $\M^{(\bl^\prime, \bmu^\prime)}$ are unitarily equivalent if and only if $\bl = \bl^\prime$ and $\bmu = \bmu^\prime$. 
\end{thm}

\begin{proof}
Since the adjoint of both $\M^{(\bl, \bmu)}$ and $\M^{(\bl^\prime, \bmu^\prime)}$ are in $\mathrm B_3(\mathbb{D}^n)$, it follows that the hermitian holomorphic vector bundles induced by $\M^{(\bl, \bmu)}$ and $\M^{(\bl^\prime, \bmu^\prime)}$ are equivalent whenever $\M^{(\bl, \bmu)}$ and $\M^{(\bl^\prime, \bmu^\prime)}$ are unitarily equivalent. Consequently, $\mathcal{K}^{ij}_{(\bl, \bmu)}(z) = \partial_i\left(K^{(\bl, \bmu)}(z,z)^{-1}{\bar{\partial}}_jK^{(\bl, \bmu)}(z,z)\right)$ and $\mathcal{K}^{ij}_{(\bl^\prime, \bmu^\prime)}(z) = \partial_i\left(K^{(\bl^\prime, \bmu^\prime)}(z,z)^{-1}{\bar{\partial}}_jK^{(\bl^\prime, \bmu^\prime)}(z,z)\right)$ are similar for all $1 \leq i, j \leq n$ and $z \in \mathbb{D}$. It can be seen that $\lambda_1-(\l_1^{-1}+\mu_1)^{-1}$, $\lambda_1$ and $\lambda_1+(\l_1^{-1}+\mu_1)^{-1}+2$ (respectively, $\lambda'_1-((\l'_1)^{-1}+\mu'_1)^{-1}$, $\lambda'_1$, and $\lambda'_1 + ((\l'_1)^{-1}+\mu'_1)^{-1}+2$) are eigenvalues of $\mathcal{K}^{11}_{(\bl, \bmu)}(0)$ (respectively $\mathcal{K}^{11}_{(\bl', \bmu')}(0)$).
Now equating the trace and determinant of $\mathcal{K}^{11}_{(\bl, \bmu)}(0)$ and $\mathcal{K}^{11}_{(\bl^\prime, \bmu^\prime)}(0)$ it can be seen that $\l_1=\l_1^{\prime}$ and $\mu_1 = \mu_1^{\prime}$, respectively. 

A similar computation shows that $\lambda_2 = \lambda^\prime_2$ and $\mu_2 = \mu_2^\prime$. Thus it follows that if $\M^{(\bl, \bmu)}$ and $\M^{(\bl^\prime, \bmu^\prime)}$ are unitarily equivalent then $\bl = \bl^\prime$ and $\bmu = \bmu^\prime$. 
\end{proof}

\subsection{Construction of Type II Hilbert spaces} Let $\a_1,\hdots,\a_n>0$ and $\ba=(\a_1,\hdots,\a_n)$. Consider the Hilbert spaces $\H_1^{(\ba)}$, $\H_2^{(\ba)}$ and $\H_3^{(\ba)}$ of holomorphic functions on $\D^n$ which are, by definition,
$$\H_1^{(\ba)}:=\otimes_{j=1}^n\A^{(\a_j+2\delta_{2j})},~\H_2^{(\ba)}:=\otimes_{j=1}^n\A^{(\a_j+2\delta_{1j})},~\text{and}~\H_3^{(\ba)}:=\otimes_{j=1}^n \A^{(\a_j+2\delta_{1j} + 2\delta_{2j})}.$$ Let $\Gamma^{(2)}_i$, $i=1,2,3$, be the linear map on the Hilbert space $\H_i^{(\ba)}$ taking values in Hol$(\D^n,\C^3)$ defined as follows:
$$\G^{(2)}_1(\bf_1)=\left(
    \begin{array}{c}
       \bf_1 \\
       0\\
       \l_1^{-1}\del_1\bf_1 \\
         \end{array}
    \right),~~~
\G^{(1)}_2(\bf_2)=\left(
    \begin{array}{c}
       0 \\
       \bf_2 \\
      \l_2^{-1}\del_2\bf_2 \\
    \end{array}
    \right),~~~\text{and}~~~
\G^{(1)}_3(\bf_3)=\left(
    \begin{array}{c}
       0 \\
       0 \\
       \bf_3  \\
    \end{array}
    \right)    
$$
where $\bf_i\in\H_i^{(\ba)}$. Note that as in case \textbf{(I)} each $\G^{(2)}_i$, $i=1,2,3$, gives rise to the Hilbert spaces $\G^{(2)}_i(\H_i^{(\ba)})$ to be denoted by $\bH_i^{(\ba)}\subset \text{Hol}(\D^n,\C^3)$. Furthermore, the point evaluation, $\G^{(2)}_i(\bf_i)\mapsto\G^{(2)}_i(\bf_i)(\bz)$, for $\bz\in\D^n$ and $i=1,2,3$, are continuous which makes $\bH_i^{(\ba)}$ a \rkhs with the reproducing kernel, say $L_i^{(\ba)}$.

Let us now consider the Hilbert space $\H^{(\ba)}(\D^n):=\oplus_{i=1}^3\H_i^{(\ba)}$ and define the linear map $\G^{(2)}:\H^{(\ba)}(\D^n)\longrightarrow \text{Hol}(\D^n,\C^3)$ as follows:
$$\G^{(2)}(\bf_1,\bf_2,\bf_3):=\sum_{i=1}^3 \G^{(2)}_i(\bf_i).$$ We note that $\G^{(2)}$ is also an one-to-one linear map onto it's image. So, for any $\b_1,\b_2>0$, as before, we define an inner product on $\G^{(2)}(\bH^{(\ba)}(\D^n))$ making it a Hilbert space, $\bH^{(\ba,\bb)}(\D^n)$, in the following manner:
$$\<\G^{(2)}(\bf_1,\bf_2,\bf_3),\G^{(2)}(\bg_1,\bg_2,\bg_3)\>:=\<\G^{(2)}_1(\bf_1),\G^{(2)}_1(\bg_1)\>+\sum_{i=2}^3\b_{i-1}^2\<\G^{(2)}_i(\bf_i),\G^{(2)}_i(\bg_i)\>.$$ Moreover, $\bH^{(\ba,\bb)}(\D^n)$ is a \rkhs with the reproducing kernel 
\beq\label{eqn1x} 
L^{(\ba,\bb)}(\bz,\bw)=L_1^{(\ba)}(\bz,\bw)+\b_1^2L_2^{(\ba)}(\bz,\bw)+\b_2^2L_3^{(\ba)}(\bz,\bw), \text{ for }\bz,\bw\in\D^n.\eeq
Suppose that $\M^{(\ba, \bb)}$ denotes the tuple of multiplication operators on the Hilbert space $\bH^{(\ba,\bb)}(\D^n)$. In the following theorem, we show that the adjoint of $\M^{(\ba, \bb)}$ also constitutes another continuous family of irreducible $n$ - tuple of operators in $\mathrm B_3(\D^n)$ which are homogeneous \w the group $\mob^n$. Since similar techniques as described in the proofs of Proposition \ref{prop2}, Proposition \ref{prop10} and Theorem \ref{thm5}, Theorem \ref{thm11}, Theorem \ref{thm12}, Theorem \ref{thm13}, Theorem \ref{thm14} with $\bA^{(\bl,\bmu)}(\D^n)$, $D^+_{\bl}$, $\M^{(\bl, \bmu)}$ replaced by $\bH^{(\ba,\bb)}(\D^n)$, $D^+_{\ba}$, $\M^{(\ba, \bb)}$ also yield the following theorem, the proof is omitted. Here $D^+_{\ba}$ is defined as $$\label{not7.1}D^+_{\ba}:=\left(D^+_{\a_1}\otimes D^+_{\a_2+2} \oplus D^+_{\a_1+2}\otimes D^+_{\a_2} \oplus D^+_{\a_1+2}\otimes D^+_{\a_2+2}\right)\otimes_{j=3}^nD^+_{\a_j}.$$

\begin{thm}\label{thm14x}
(i) The reproducing kernel of the Hilbert space $\bH^{(\ba,\bb)}(\D^n)$ is 
\begin{flalign*}
L^{(\ba,\bb)}(\bz,\bw) &:=
    D_2(zw)\begin{pmatrix}
    1 & 0 & z_1\\
    0 & \b_1^2 & \b_1^2z_2\\
    \ov{w}_1 & \b_1^2\ov{w}_2 & z_1\ov{w}_1+\b_1^2z_2\ov{w}_2 + \frac{1}{\a_1}+\frac{\b_1^2}{\a_2}+\b_2^2
    \end{pmatrix}D_2(zw)
\\
  &\times\prod_{j=1}^n(1-z_j\ov{w}_j)^{-\a_j-2\delta_{1j}-2\delta_{2j}}
\end{flalign*}
where $D_2(\bz,\bw)$ is the $3\times 3$ diagonal matrix diag$((1-z_1\ov{w}_1),(1-z_2\ov{w}_2),1)$ and $\delta$ is the Kronecker delta function.
  
(ii) The image of $D^+_{\ba}$ under the map $\G^{(2)}$ is a multiplier representation with the multiplier given by 
\[J^{(2)}(g,\bz):=  \begin{pmatrix}
      (g'_1)^{\frac{\a_1}{2}}(g'_2)^{\frac{\a_2+2}{2}} & 0 & 0\\
 0     & (g'_1)^{\frac{\a_1+2}{2}}(g'_2)^{\frac{\a_2}{2}} & 0 \\
-c_{g_1}(g'_1)^{\frac{\a_1+1}{2}}(g'_2)^{\frac{\a_2+2}{2}} & -c_{g_2}(g'_1)^{\frac{\a_1+2}{2}}(g'_2)^{\frac{\a_2+1}{2}}  & (g'_1)^{\frac{\a_1+2}{2}}(g'_2)^{\frac{\a_2+2}{2}} \\
    \end{pmatrix}\prod_{j=3}^n (g'_{j})^{\frac{\a_j}{2}}(\bz)\]
 where $c_{g_i}$ is determined by the equation \eqref{eqn9}.
 
(iii) For any $n$-tuple of positive real numbers $\ba = (\a_1, \a_2, \ldots, \a_n)$ and $\b_1, \b_2 > 0$, $\M^{(\ba,\bb)}$ is bounded, homogeneous, irreducible and unitarily inequivalent to the $n$-tuple of multiplication operators on the reproducing kernel Hilbert space $\H_{\hat{K}}$.

(iv) The $n$-tuple of multiplication operators $\M^{(\ba, \bb)}$ and $\M^{(\ba^\prime, \bb^\prime)}$ on $\bH^{(\ba,\bb)}(\D^n)$ and $\bH^{(\ba',\bb')}(\D^n)$, respectively are unitarily equivalent if and only if $\ba= \ba^\prime$ and $\bb = \bb^\prime$. 
\end{thm}

\section{Classification of irreducible homogeneous tuples in $\mathrm B_3(\D^n)$}\label{Section 8}

In this section, we describe, up to unitary equivalence, all $n$ - tuples of operators in $\mathrm B_3(\D^n)$  homogeneous \w the group  of bi-holomorphic automorphisms on $\D^n$ as well as the identity component of the group $\aut(\D^n)$, namely, $\mob^n$. It is proved that these are the $n$ - tuples of multiplication by the coordinate functions on the {\rkhs} $\H_K$, where the reproducing kernel $K$ is either the reproducing kernel $\hat{K}$ or the reproducing kernel $K^{(\bl,\bmu)}$ obtained in the Proposition \ref{prop2} or the reproducing kernel $L^{(\ba,\bb)}$ obtained in the Theorem \ref{thm14x}(i).

We also show that there are no $\aut(\D^n)$ - homogeneous $n$-tuple of operators in $\mathrm B_3(\D^n)$ whenever $n \geq 3$. For $n=2$, the pair of operators $(M_{z_1}^*, M_{z_2}^*)$ on the Hilbert spaces $\bA^{(\bl,\bmu)}(\D^2),\,\,\l_1 = \l_2,\,\,\mu_1 = \mu_2$ and $\bH^{(\ba,\bb)}(\D^n),\,\,\a_1 = \a_2,\,\,\b_1 = 1$ are the only pairs in $\mathrm B_3(\D^2)$ that are homogeneous {\w} the group $\aut(\D^2)$.

\subsection{Homogeneity {\w}$\mob^n$.}

In this subsection, we study all $n$ - tuples of operators in $\mathrm B_3(\D^n)$ which are homogeneous under the action of $\mob^n$. We first describe all possible $3$ dimensional indecomposable representations of the Lie algebra $\mathfrak{b}^n$ which determine the cocycles associated to an irreducible $n$ - tuple of operators in $\mathrm B_3(\D^n)$ homogeneous \w the group $\mob^n$. 

Let $k \leq n$ be an arbitrary but fixed integer. We identify the Lie algebra $\mathfrak{b}^{k}$ with the subalgebra $\mathfrak{b}^{k}\oplus \{0\}\oplus\cdots\oplus\{0\}$ of $\mathfrak{b}^{n}$. Note that to verify the indecomposability of a representation $\rho$ of $\mathfrak{b}^{n}$, it is enough to consider its restriction to $\mathfrak{b}^{k}$ for some $k \leq n$,  provided the representation obtained by restricting $\rho$ to $\mathfrak{b}^k$ is multiplicity-free, as explained in the following lemma.

\begin{lem}\label{criterion for indecomposability}
Suppose that $\rho:\mathfrak{b}^{n}\ra \mathfrak{gl}(r,\C)$ is a representation such that the representation $\rho':\mathfrak{b}^{k}\ra\mathfrak{gl}(r,\C)$ defined by $\rho':=\rho|_{\mathfrak{b}^k}$ is multiplicity-free for some $k\leq n$. Then $\rho$ and $\rho'$ are simultaneously indecomposable or decomposable.
\end{lem}
\begin{proof}
We begin by pointing out that if $\rho$ is decomposable, so is $\rho^\prime$. Conversely, assume that $\rho'$ is decomposable. Then there exist two non-trivial subspcaes $V$ and $W$ such that $\C^r=V\oplus W$ and both $V$ and $W$ are invariant under $\rho'$. Since $\rho'$ is multiplicity-free there are scalars $\a_1,\hdots,\a_k$ such that the operator $T:=\sum_{i=1}^k\a_i\rho(h_i)$ has distinct eigenvalues. Also, the joint eigenvectors of $(\rho(h_1),\hdots,\rho(h_k))$ are eigenvectors of $T$ and both $V$ and $W$ remain invariant under $T$.


Since $T$ is diagonal and $T$ has distinct eigenvalues, it follows that both $V$ and $W$ are spanned by eigenvectors of $T$. The eigenvectors of $T$ are exactly those of each $\rho(h_i)$, $1\leq i\leq n$ since $\rho(h_1),\hdots,\rho(h_n)$ are simultaneously diagonalizable. Consequently, both $V$ and $W$ are invariant under each $\rho(h_i)$, $i=1,\hdots,n$. 

For $k+1\leq j\leq n$, $\rho(y_j)$ commutes with each $\rho(h_i)$, $1\leq i\leq k$ implying that $\rho(y_j)$ commutes with $T$. Furthermore, since $T$ is a diagonal matrix with distinct diagonal entries, it follows that $\rho(y_j)$ is diagonal for each $k+1\leq j\leq n$. Thus $\rho(h_j)$ and $\rho(y_j)$ commute for each $k+1\leq j\leq n$ which together with the relation $[\rho(h_j),\rho(y_j)]=-\rho(y_j)$ yield that $\rho(y_j)=0$ for each  $k+1\leq j\leq n$ verifying that both $V$ and $W$ remain invariant under $\rho(y_j)$, $j=k+1,\hdots, n$. Thus it follows that $\rho$ is decomposable.
\end{proof}

Lemma \ref{criterion for indecomposability} actually states that an indecomposable representation $\rho$ of $\mathfrak{b}^n$ turns out to be the tensor product of it's restriction to $\mathfrak{b}^k$ for $k\leq n$ with the one dimensional representations of the remaining factors in $\mathfrak{b}^n$ whenever the restriction of $\rho$ to $\mathfrak{b}^k$ is already multiplicity-free.

%

\begin{prop}\label{description of rank 3 representation}
Suppose that $\rho:\mathfrak{b}^n\ra \mathfrak{gl}(r,\C)$ is an indecomposable representation and that $\rho|_{\mathfrak{b}^1}$ is decomposable. Then up to equivalence the following statements hold.
\begin{itemize}
\item[(i)] $\rho|_{\mathfrak{b}^2}$ is multiplicity-free and consequently, indecomposable.
\item[(ii)] $\rho(y_i)=0$ for $i\geq 3$.
\item[(iii)] $\rho(h_i)=\l_iI_3$ for $i\geq 3$, $\l_i$ are scalars.
\end{itemize}
\end{prop}

\begin{proof}
We begin with the observation that indecomposability of $\rho$ implies $\rho(y_i)\neq 0$ for some $1\leq i\leq n$. Without loss of generality assume that $\rho(y_1)\neq 0$. So $\rho(h_1)$ can not be scalar times the identity matrix. We prove that there exists $1\leq i\leq n$ such that $(\rho(h_1),\rho(h_i))$ has distinct joint eigenvalues.


If possible suppose that for any $1\leq i\leq n$, $(\rho(h_1),\rho(h_i))$ does not have distinct joint eigenvalues. Let $(\a_1,\b_1), (\a_2,\b_2)$ and $(\a_3,\b_3)$ be joint eigenvalues of $(\rho(h_1),\rho(h_i))$ such that $(\a_1,\b_1)=(\a_2,\b_2)$.

Since $\rho(y_1)\neq 0$ it follows that $\rho(y_1)$ maps a joint eigenvector corresponding to the joint eigenvalue $(\a_1,\b_1)$ to a joint eigenvector corresponding to the joint eigenvalue $(\a_3,\b_3)$ or vice verse. But the identities $[\rho(h_1),\rho(y_1)]=-\rho(y_1)$ and $[\rho(h_i),\rho(y_1)]=0$ together imply that either $\rho(y_1)$ maps a joint eigenvector corresponding to the joint eigenvalue $(\a_1,\b_1)$ to a joint eigenvector associated to the joint eigenvalue $(\a_1-1,\b_1)$ or a joint eigenvector corresponding to the joint eigenvalue $(\a_3,\b_3)$ to a joint eigenvector associated to the joint eigenvalue $(\a_3-1,\b_3)$. In either of these two cases, we have that $\b_1=\b_2=\b_3$ verifying that $\rho(h_i)=\b_1I_3$. It turns out from the identity $[\rho(h_i),\rho(y_i)]=-\rho(y_i)$ that $\rho(y_i)=0$. Thus for every $2\leq i\leq n$, $\rho(h_i)=\l_iI_3$ and $\rho(y_i)=0$ for some scalars $\l_2,\hdots,\l_n$ which contradicts the hypothesis that $\rho$ is indecomposable since $\rho|_{\mathfrak{b}^1}$ is already decomposable. Now, statements (ii) and (iii) follow from the proof of the Lemma \ref{criterion for indecomposability}.
\end{proof}
An immediate corollary of this proposition is stated below. The proof follows by combining Theorem \ref{Classification of rep} and Proposition \ref{description of rank 3 representation}.
\begin{cor}\label{thm15}
If $\rho:\mathfrak{b}^n\ra\mathfrak{gl}(3,\C)$ is an idecomposable representation, then $\rho$ is one of the following:
\begin{itemize}
\item[(i)] $\rho|_{\mathfrak{b}^1}$ is indecomposable and $\rho(h_i)=\l_iI_3$, $\rho(y_i)=0$, for $i=2,3,\hdots,n$ where $\l_i$ are scalars.
\item[(ii)] $$\rho(h_1)=
    \begin{pmatrix}
      \l_1 & 0 & 0 \\
      0 & \l_1-1    &   0 \\
      0 & 0 & \l_1 \\
    \end{pmatrix}
    ,
\rho(y_1)=
    \begin{pmatrix}
      0 & 0 & 0 \\
      1 & 0    &   0 \\
      0 & 0 & 0 \\
    \end{pmatrix}
    ,$$ 
$$\rho(h_2)=
    \begin{pmatrix}
      \l_2 & 0 & 0 \\
      0 & \l_2    &   0 \\
      0 & 0 & \l_2-1 \\
    \end{pmatrix}
    ,
\rho(y_2)=
    \begin{pmatrix}
      0 & 0 & 0 \\
      0 & 0    &   0 \\
      1 & 0 & 0 \\
    \end{pmatrix}
    $$
and $\rho(h_i)=\l_iI_3$, $\rho(y_i)=0$, for $i=3,\hdots,n$, where $\l_i$ are all scalars.
\item[(iii)] $$\rho(h_1)=
    \begin{pmatrix}
      \a_1-1 & 0 & 0 \\
      0 & \a_1    &   0 \\
      0 & 0 & \a_1-1 \\
    \end{pmatrix},\,\,
\rho(y_1)=
    \begin{pmatrix}
      0 & 0 & 0 \\
      0 & 0    &   0 \\
      0 & 1 & 0 \\
    \end{pmatrix}
    ,$$ 
$$\rho(h_2)=
    \begin{pmatrix}
      \a_2 & 0 & 0 \\
      0 & \a_2-1    &   0 \\
      0 & 0 & \a_2-1 \\
    \end{pmatrix},\,\,
\rho(y_2)=
    \begin{pmatrix}
      0 & 0 & 0 \\
      0 & 0    &   0 \\
      1 & 0 & 0 \\
    \end{pmatrix},
    $$
and $\rho(h_i)=\a_iI_3$, $\rho(y_i)=0$, for $i=3,\hdots,n$, where $\a_i$ are all scalars.
\end{itemize}
\end{cor}

We are now in a position to describe all cocycles on $\widetilde{\text{SU}(1, 1)^n} \times \mathbb{D}^n$. For this, we first recall that any cocycle $\hat{J}$ on $\widetilde{\text{SU}(1,1)}\times \D$ taking values in GL$(3,\C)$ such that $\hat{J}(\tilde{k},0)$ is diagonal for all $\tilde{k}$ in the stabilizer subgroup of the origin, is of the form 
\beq\label{eqn15}
 \hat{J}(\tilde{g},z):=
    \begin{pmatrix}
      (g')^{\frac{\l}{2}} & 0 & 0\\
 -2c_{g}(g')^{\frac{\l+1}{2}}     & (g')^{\frac{\l+2}{2}} & 0 \\
-3c_{g}^2(g')^{\frac{\l+2}{2}} & -3c_g(g')^{\frac{\l+3}{2}} & (g')^{\frac{\l+4}{2}} \\
    \end{pmatrix},
\eeq for some $\l>0$ where $g=p(\tilde{g})$ and $p:\widetilde{\mob}\ra \mob$ is the universal covering map, see \cite{ACHOCD}.

\begin{prop}\label{thm 16}
Suppose $J : \widetilde{\text{SU}(1, 1)}^n \times \mathbb{D}^n \to \text{GL}(3, \mathbb{C})$ is  a cocycle such that $J(\tilde{k}, 0)$ is diagonal for all $\tilde{k} \in \tilde{\K}$. Then $J(\tilde{g}, \bz)$ takes one of the following form.
\begin{itemize}
\item[(i)] $\hat{J}(\tilde{g}_1,z_1)\prod_{j=2}^{n}(g_j'(z_j))^{\l_j}$ where $\hat{J}$ is the cocycle as in \eqref{eqn15}.

\item[(ii)] $\left(\!\!
    \begin{smallmatrix}
      (g'_1)^{\frac{\l_1}{2}}(g'_2)^{\frac{\l_2}{2}} & 0 & 0\\
 -c_{g_1}(g'_1)^{\frac{\l_1+1}{2}}(g'_2)^{\frac{\l_2}{2}}      & (g'_1)^{\frac{\l_1+2}{2}}(g'_2)^{\frac{\l_2}{2}} & 0 \\
-c_{g_2}(g'_1)^{\frac{\l_1}{2}}(g'_2)^{\frac{\l_2+1}{2}} & 0 & (g'_1)^{\frac{\l_1}{2}}(g'_2)^{\frac{\l_2+2}{2}} \\
    \end{smallmatrix}\!\!
    \right)\prod_{j=3}^n (g'_{j})^{\frac{\l_j}{2}}(\bz)$
    where $c_{g_i}$ is determined by the equation \eqref{eqn9}.
\vspace{0.05in}
\item[(iii)] $\left(\!\!
    \begin{smallmatrix}
      (g'_1)^{\frac{\a_1}{2}}(g'_2)^{\frac{\a_2+2}{2}} & 0 & 0\\
 0     & (g'_1)^{\frac{\a_1+2}{2}}(g'_2)^{\frac{\a_2}{2}} & 0 \\
-c_{g_1}(g'_1)^{\frac{\a_1+1}{2}}(g'_2)^{\frac{\a_2+2}{2}} & -c_{g_2}(g'_1)^{\frac{\a_1+2}{2}}(g'_2)^{\frac{\a_2+1}{2}}  & (g'_1)^{\frac{\a_1+2}{2}}(g'_2)^{\frac{\a_2+2}{2}} \\
    \end{smallmatrix}\!\!
    \right)\prod_{j=3}^n (g'_{j})^{\frac{\a_j}{2}}(\bz)$
 where $c_{g_i}$ is determined by the equation \eqref{eqn9}.
\end{itemize} 
\end{prop} 

\begin{proof}
We begin by pointing out that, for any cocycle $J : \widetilde{\text{SU}(1, 1)}^n \times \mathbb{D}^n \to \text{GL}(3, \mathbb{C})$, there is an indecomposable representation $\rho:\mathfrak{b}^n\ra\mathfrak{gl}(3,\C)$ such that $$J(\tilde{g},\bz)=\rho(s(\bz)^{-1}\tilde{g}^{-1}s(\tilde{g}\cdot\bz))$$ where $s:\D^n\ra (\mathbb{S}^2)^n$ is a holomorphic section. Since $\rho(k,0)$ is diagonal for every $k\in\K$, it follows that $\rho$ is diagonalizable on the subalgebra spanned by $\{h_j:1\leq j\leq n\}$. Thus, $\rho$ is one of the representations obtained in Theorem \ref{thm14}.

Let $\tilde{g}\in\widetilde{\text{SU}(1,1)^n}$ with $p(\tilde{g})= g =\left(\! \!\left(\!\begin{smallmatrix}
a_i & b_i\\
c_i & d_i
\end{smallmatrix} \!\right)\!\!\right)_{i=1}^{n} \in \text{SU}(1,1)^n$. If $\rho$ is of the form $(i)$ (respectively, $(ii)$ or $(iii)$) in Corollary \ref{thm15} substituting the values of $\rho(h_i)$ and $\rho(y_i)$ in \eqref{cocyclecomp3} we obtain that $J$ takes the form $(i)$ (respectively, $(ii)$ or $(iii)$).
\end{proof}

In the theorem below, we describe all  $\mob^n$ - homogeneous $n$ - tuples of operators in $\mathrm B_3(\D^n)$ with the help of the classification of all cocycles of $\widetilde{\text{SU}(1,1)^n}\times \D^n$ obtained above. 

\begin{thm}\label{thm17}
Let $\M^*=(M_{z_1}^*,\hdots,M_{z_n}^*)\in\mathrm B_3(\D^n)$ and suppose that $\M$ is irreducible and homogeneous \w the group $\mob^n$. Then $\M$ is unitarily equivalent to either the $n$ - tuple of multiplication operators  on the \rkhs $\H_{\hat{K}}$ or the $n$ - tuple $\M^{(\bl,\bmu)}$ acting on $\bA^{(\bl,\bmu)}(\D^n)$ or the $n$ - tuple $\M^{(\ba,\bb)}$ acting on $\bH^{(\ba,\bb)}(\D^n)$.
\end{thm}

\begin{proof}
We recall from \cite[Theorem 3.1]{OICHO} that since $\M^*$ is an irreducible homogeneous operator in $\mathrm B_3(\D^n)$ there exists a cocycle $J:\widetilde{\text{SU}(1,1)^n}\times \D^n\ra \text{GL}(3,\C)$ satisfying 
\beq\label{eqn18}
K(\bz,\bw) &=& J(\tilde{g},\bz)K(\tilde{g}\bz,\tilde{g}\bw)J(\tilde{g},\bw)^*,
\eeq 
for $\tilde{g}\in\widetilde{\text{SU}(1,1)^n}$ with $p(\tilde{g})=g\in$ $\mob^n$ and $\bz,\bw\in\D^n$. Note that $J$ is one of form (i), (ii) or (iii) of the Proposition \ref{thm 16}. If $J$ is of the form Proposition \ref{thm 16}(i), then $K$ is equivalent to the kernel $\hat{K}$ where $\hat{K}$ is defined as in the Equation \eqref{eqn14}.

Now suppose that $J$ takes the form described in Proposition \ref{thm 16}(ii).
For $\tilde{k}\in\tilde{\K}$, it follows from the Equation \eqref{eqn18} that $J(\tilde{k},0) $ commutes with $K(0,0)$. Consequently, $K(0,0)$ is a diagonal matrix. Let $K(0,0)=\text{diag}(1, d_1,d_2)$. Now since $K$ is a positive semi-definite and $M_{z_i}$ are bounded for all $1 \leq i \leq n$,
we have, for $3\leq i\leq n$, that $\l_i>0$.

Clearly, for two points $\bz_1=(0,\hdots,0)$ and $\bz_2=(z,0\hdots,0)$, $z\in\D$, the matrix $P(z)=(\!(K(\bz_1,\bz_2))\!)_{i,j=1}^2$ is positive semi-definite. So it follows that the submatrix $Q(z)$ of $P(z)$ defined as 
\[ Q(z):=
    \left(
    \begin{array}{ccc}
      \<P(z)\varepsilon_1,\varepsilon_1\> & \<P(z)\varepsilon_1,\varepsilon_2\> & \<P(z)\varepsilon_1,\varepsilon_4\>\\
 \<P(z)\varepsilon_2,\varepsilon_1\>     & \<P(z)\varepsilon_2,\varepsilon_2\> & \<P(z)\varepsilon_2,\varepsilon_4\>  \\
\<P(z)\varepsilon_4,\varepsilon_1\>  & \<P(z)\varepsilon_4,\varepsilon_2\> & \<P(z)\varepsilon_4,\varepsilon_4\> \\
    \end{array}
    \right)\]
is positive semidefinite where $\{\varepsilon_i:1\leq i\leq 6\}$ is the standard ordered basis for $\C^6$. Consequently, $\text{det}(Q(z))\geq 0$ for all $z\in\D$. A simple calculation yields that $$\text{det}(Q(z))=d_1(1-|z|^2)^{-\l_1}-|z|^2-d_1.$$ Let us consider the smooth function $f(r)=d_1(1-r)^{-\l_1}-r-d_1$. Note that $f(r)\geq 0$ for $r\in(0,1)$ and $f(0) = 0$. This implies that $f'(0)\geq 0$ from which it follows that $d_1\geq \frac{1}{\l_1}$. A similar calculation yields that $d_2\geq \frac{1}{\l_2}$.

Now it remains to show that $d_i\neq \frac{1}{\l_i}$ for $i=1,2$. We prove that the multiplication operator $M_{z_1}$ can not be bounded whenever $d_1=\frac{1}{\l_1}$. On the contrary, suppose that there is a  positive constant $C$ such that $$K_C(\bz,\bw)=(C-z_1\ov{w}_1)K(\bz,\bw),~~~\bz,\bw\in\D^n$$ is a non-negative definite kernel. As before let $\bz_1=(0,\hdots,0)$ and $\bz_2=(z,0\hdots,0)$ with $z\in\D$ be two points in $\D^n$. Since $K_C$ is a non-negative definite kernel the matrix $P_C(z)=(\!(K_C(\bz_i,\bz_j))\!)_{i,j=1}^2$ is also a non-negative definite matrix. Therefore, the submatrix
\[ Q_C(z):=
    \left(
    \begin{array}{ccc}
      \<P_C(z)\varepsilon_1,\varepsilon_1\> & \<P_C(z)\varepsilon_1,\varepsilon_2\> & \<P_C(z)\varepsilon_1,\varepsilon_4\>\\
 \<P_C(z)\varepsilon_2,\varepsilon_1\>     & \<P_C(z)\varepsilon_2,\varepsilon_2\> & \<P_C(z)\varepsilon_2,\varepsilon_4\>  \\
\<P_C(z)\varepsilon_4,\varepsilon_1\>  & \<P_C(z)\varepsilon_4,\varepsilon_2\> & \<P_C(z)\varepsilon_4,\varepsilon_4\> \\
    \end{array}
    \right)\] 
is non-negative. In particular, for $z\in\D$, 
$$\text{det}(Q_C(z))=\dfrac{C^2}{\l_1}(C-|z|^2)(1-|z|^2)^{-\l_1}-C^3|z|^2-\dfrac{C^3}{\l_1}\geq 0.$$ 
Thus substituting $|z|^2=r$ we have that the function 
$$f_C(r)=\dfrac{C^2}{\l_1}(C-r)(1-r)^{-\l_1}-C^3r-\dfrac{C^3}{\l_1}\geq 0$$ 
for all $r\in [0,1)$. Also, observe that $f_C(0)=0$ and $$\lim_{h\rightarrow 0^+}\dfrac{f_C(h)}{h}=-\dfrac{C^2}{\l_1}<0.$$ As a consequence, there is a $r\in(0,1)$ such  that $f_C(r)<0$. This contradicts the fact that $f_C(r)\geq 0$ on $[0,1)$. Therefore, $K_C$ cannot be positive definite kernel for any choice of positive $C$ which shows that $M_{z_1}$ is not bounded. Thus we must have $d_1>\frac{1}{\l_1}$. In a similar way, it can be shown that $d_2>\frac{1}{\l_2}$.

For $i=1,2$, we have shown that $d_i>\frac{1}{\l_i}$. In consequence, there exist real numbers $\mu_i>0$, $i=1,2$, such that $d_i=\mu_i+\frac{1}{\l_i}$. This proves that $K=K^{(\bl,\bmu)}$ and therefore, $\H_K=\bA^{(\bl,\bmu)}(\D^n)$.

Finally, suppose $J$ takes the form described in Proposition \ref{thm 16}(iii). Let $K(0,0)=\text{diag}(1, d_1,d_2)$. A similar argument given above yields that $\a_i > 0$, for $3 \leq i \leq n$. Now, considering the submatrix
\[ \tilde{Q}(z):=
    \left(
    \begin{array}{ccc}
      \<P(z)\varepsilon_1,\varepsilon_1\> & \<P(z)\varepsilon_1,\varepsilon_3\> & \<P(z)\varepsilon_1,\varepsilon_4\>\\
 \<P(z)\varepsilon_3,\varepsilon_1\>     & \<P(z)\varepsilon_3,\varepsilon_3\> & \<P(z)\varepsilon_3,\varepsilon_4\>  \\
\<P(z)\varepsilon_4,\varepsilon_1\>  & \<P(z)\varepsilon_4,\varepsilon_3\> & \<P(z)\varepsilon_4,\varepsilon_4\> \\
    \end{array}
    \right)\]
and following a similar technique as above implies that $d_2 \geq \frac{1}{\a_1}.$ If $d_2 = \frac{1}{\a_1}$, again following a similar argument as before, we prove that $M_{z_1}$ is not bounded. Thus, there exists $\b_1, \b_2 > 0$ such that $K = L^{(\ba,\bb)}$ and therefore $\mathcal{H}_K = \bH^{(\ba,\bb)}$.    
\end{proof}

\begin{rem}\label{xxx}
Note that the cocyles given in Proposition \ref{thm 16}(ii) and (iii) are inequivalent for any $\bl, \bmu, \ba, \bb$ and therefore, the operators $\M^{(\bl,\bmu)}$ and $\M^{(\ba, \bb)}$ are always inequivalent.
\end{rem}

\subsection{Homogeneity \w Aut$(\D^n)$}

We now consider homogeneity of $n$ - tuples of operators in $\mathrm B_3(\D^n)$ under the action of $\aut(\D^n)$. It is shown that for $n=2$, there are irreducible $\aut(\D^2)$ - homogeneous operator tuples in $\mathrm B_3(\D^2)$ unlike the case of rank $2$ where there is no $\aut(\D^n)$ - homogeneous $n$ - tuples of operators in $\mathrm B_2(\D^n)$ for any $n\geq 2$. However, there are no homogeneous operators in $\mathrm B_3(\D^n)$ \w the group $\aut(\D^n)$ if $n\geq 3$. 

\begin{prop}\label{thm21}
For $n\geq 3$, $\M^{(\bl,\bmu)}$ and $\M^{(\ba, \bb)}$ are not homogeneous under the action of the group $\aut(\D^n)$.
\end{prop}

\begin{proof}
We prove that if $n \geq 3$, then $\M^{(\bl, \bmu)}$ is not homogeneous with respect to $\aut(\D^n)$. The proof for $\M^{(\ba, \bb)}$ is similar. 

Since $\l_1$, $\frac{\l_1^2\mu_1}{\l_1\mu_1 - 1}$ and $2+\frac{\l_1(2-\l_1\mu_1)}{1-\l_1\mu_1}$ are eigenvalues of $\mathcal{K}^{11}(0)$ and on the other hand, $\l_i$ is the only eigenvalue of $\mathcal{K}^{ii}(0)$, for $i\geq 3$, the proof follows from Lemma \ref{curvature at 0 is diagonal}.
\end{proof}
The proof of the following theorem is similar to that of Proposition \ref{thm21}, therefore omitted. 
\begin{prop}\label{thm22}
For $n\geq 2$, the $n$ - tuple of multiplication operators $\M=(M_{z_1},\hdots,M_{z_n})$ defined on the \rkhs $\H_{\hat{K}}$ is not homogeneous under the action of the group $\aut(\D^n)$, where $\hat{K}$ is defined in Equation  \eqref{eqn14}.
\end{prop}

Finally, in the following theorem we show that the tuple of multiplication operators by coordinate functions on the \rkhs $\bA^{(\bl,\bmu)}(\D^2)$ is homogeneous with respect to $\aut(\D^2)$ under certain condition..

\begin{thm}\label{thm19}
(i) The pair of multiplication operators $(M_{z_1},M_{z_2})$ on the \rkhs $\bA^{(\bl,\bmu)}(\D^2)$ with the reproducing kernel $K^{(\bl,\bmu)}$ on $\D^2$, $\bl=(\l_1,\l_2)$ and $\bmu=(\mu_1,\mu_2)$, is homogeneous under the action of the group $\aut(\D^2)$ if and only if $\l_1=\l_2$ and $\mu_1=\mu_2$.

(ii) The pair of multiplication operators $(M_{z_1},M_{z_2})$ on the \rkhs $\bH^{(\ba,\bb)}(\D^2)$ with the reproducing kernel $L^{(\ba,\bb)}$ on $\D^2$, $\ba=(\a_1,\a_2)$ and $\bb=(\b_1,\b_2)$, is homogeneous under the action of the group $\aut(\D^2)$ if and only if $\a_1=\a_2$ and $\b_1=1$.
\end{thm}

\begin{proof}
We prove the first part of the theorem. The proof of the second part is similar and therefore omitted. Let us begin by pointing out that $(M_{z_1},M_{z_2})$ is homogeneous \w the group $\mob^2$. We first consider the converse direction. So assume that $\l_1=\l_2$ and $\mu_1=\mu_2$. Since $\aut(\D^2)$ is the semi-direct product of $\mob^2$ and the permutation group $S_2$ of two elements, it is enough to show that $\si(M_{z_1},M_{z_2})=(M_{z_2},M_{z_1})$ is unitarily equivalent to $(M_{z_1},M_{z_2})$ where $\si\in S_2$ is the non-trivial element of $S_2$. This is equivalent to the fact that the pair of multiplication operators on the \rkhs $\H_{\si}$ with the reproducing kernel 
\beq\label{eqn20}
K_{\si}((z_1,z_2),(w_1,w_2)) := K((z_2,z_1),(w_2,w_1)), ~~~~z_i,w_i\in\D,i=1,2, 
\eeq  
is unitarily equivalent to $(M_{z_1},M_{z_2})$. Indeed, the matrix $A$ of the linear operator $T_{\si}\in \text{GL}(3,\C)$ defined by $T_{\si}(\varepsilon_1)=\varepsilon_1$ and $T_{\si}(\varepsilon_j)=\varepsilon_{\si(j)}$, $j=2,3$, satisfies the following equation 
$$K_{\si}(\bz,\bw)=AK(\bz,\bw)A^*,~~~~\bz=(z_1,z_2),\bw=(w_1,w_2)\in\D^2.$$
For the forward direction, assume that $(M_{z_1},M_{z_2})$ is homogeneous \w the group $\aut(\D^2)$. Since $(M_{z_1},M_{z_2})\in \mathrm B_3(\D^2)$, there exists a cocycle $J:\aut(\D^2)\times \D^2\ra \text{GL}(3,\C)$ such that $K^{(\bl,\bmu)}$ is quasi-invariant \w $J$. It follows that from Lemma \ref{curvature at 0 is diagonal} that $\text{tr }\mathcal{K}^{11}(0)=\text{tr }\mathcal{K}^{22}(0)$ and $\text{det }\mathcal{K}^{11}(0)=\text{det }\mathcal{K}^{22}(0)$. Consequently, we have that $\l_1=\l_2$ and $\mu_1=\mu_2$ since 
$\l_i$, $\frac{\l_i^2\mu_i}{\l_i\mu_i - 1}$ and $2+\frac{\l_i(2-\l_i\mu_i)}{1-\l_i\mu_i}$ are eigenvalues of $\mathcal{K}^{ii}$, $i=1,2$.
\end{proof}

\begin{rem}
(i) If a commuting $n$ - tuple of operators $\T=(T_1,\hdots,T_n)$ is homogeneous \w $\aut(\D^n)$, then $\T$ must be homogeneous under the action of the subgroup $\mob^n$ as well. Consequently, combining Proposition \ref{thm21}, Proposition \ref{thm22} and Theorem \ref{thm19} we see, for $n\geq 3$, that there are no irreducible homogeneous $n$ - tuple of operators in $\mathrm B_3(\D^n)$ \w $\aut(\D^n)$.

\vspace{0.05in}

(ii) For $n=2$, Proposition \ref{thm22} and Theorem \ref{thm19} together imply, up to unitary equivalence, that the pair of multiplication operators by coordinate functions on the reproducing kernel Hilbert spaces $\bA^{(\bl,\bmu)}(\D^2)$ with $\bl=(\l,\l)$ and $\bmu=(\mu,\mu)$ and $\bH^{(\ba,\bb)}$ with $\ba = (\a_1,\a_2)$ and $\bb=(1,\b_2)$ are the only irreducible homogeneous tuples of operators in $\mathrm B_3(\D^2)$ under the action of $\aut(\D^2)$.

\end{rem}

\medskip \textit{Acknowledgement}.
The authors are indebted to Professor Adam Kor{\'a}nyi for very detailed comments and suggestions.
They would also like to express their sincere gratitude to Professor Gadadhar Misra for his patient guidance and suggestions in the preparation of this paper.

\end{document}